\documentclass[11pt,reqno]{amsart}
\usepackage{amsfonts, amsmath, amssymb, amscd, amsthm, bm}
\usepackage{enumerate}
\usepackage{url}
\usepackage[linktocpage=true,colorlinks,citecolor=magenta,linkcolor=blue,urlcolor=magenta]{hyperref}
\usepackage{multicol}
\usepackage{cite}
\usepackage{comment}
\usepackage{graphicx}
\usepackage{xcolor}
\usepackage{amsrefs}

 \newtheorem{thm}{Theorem}[section]
 \newtheorem{cor}[thm]{Corollary}
  
 \newtheorem{lem}[thm]{Lemma}
 \newtheorem{assump}[thm]{Assumption}
 \newtheorem{prop}[thm]{Proposition}
 \theoremstyle{definition}
 
 \theoremstyle{remark}
 \newtheorem{rem}[thm]{Remark}
 \newtheorem*{ack}{Acknowledgments}

 \theoremstyle{claim}
 \numberwithin{equation}{section}
 \newcommand{\RR}{{\mathbb R}}
 \newcommand{\E}{{\mathrm e}}
 \newcommand{\e}{{\mathbf e}}

\begin{document}
\title[Volume preserving flow]{Volume preserving flow and Alexandrov-Fenchel type inequalities in hyperbolic space}
\author[B. Andrews]{Ben Andrews}
\address{Mathematical Sciences Institute,
Australian National University,
ACT 2601 Australia}
\email{\href{mailto:Ben.andrews@anu.edu.au}{Ben.andrews@anu.edu.au}}
\author[X. Chen]{Xuzhong Chen}
\address{College of Mathematics and Econometrics, Hunan University, Changsha, P. R. China}
\email{\href{mailto:chenxuzhong@hnu.edu.cn}{chenxuzhong@hnu.edu.cn}}
\author[Y. Wei]{Yong Wei}
\address{Mathematical Sciences Institute,
Australian National University,
ACT 2601 Australia}
\email{\href{mailto:yong.wei@anu.edu.au}{yong.wei@anu.edu.au}}

\date{\today}
\thanks{This research was supported by Australian Laureate Fellowship FL150100126 of the Australian Research Council. The second author was also supported by the Fundamental Research Funds for the Central Universities.}


\begin{abstract}
In this paper, we study flows of hypersurfaces in hyperbolic space, and apply them to prove geometric inequalities. In the first part of the paper, we consider volume preserving flows by a family of curvature functions including positive powers of $k$-th mean curvatures with $k=1,\cdots,n$, and positive powers of $p$-th power sums $S_p$ with $p>0$.  We prove that if the initial hypersurface $M_0$ is smooth and closed and has positive sectional curvatures, then the solution $M_t$ of the flow has positive sectional curvature for any time $t>0$, exists for all time and converges to a geodesic sphere exponentially in the smooth topology. The convergence result can be used to show that certain Alexandrov-Fenchel quermassintegral inequalities, known previously for horospherically convex hypersurfaces, also hold under the weaker condition of positive sectional curvature.

In the second part of this paper, we study curvature flows for strictly horospherically convex hypersurfaces in hyperbolic space with speed given by a smooth, symmetric, increasing and homogeneous degree one function $f$ of the shifted principal curvatures $\lambda_i=\kappa_i-1$, plus a global term chosen to impose a constraint on the quermassintegrals of the enclosed domain, where $f$ is assumed to satisfy a certain condition on the second derivatives.  We prove that if the initial hypersurface is smooth, closed and strictly horospherically convex, then the solution of the flow exists for all time and converges to a geodesic sphere exponentially in the smooth topology. As applications of the convergence result,  we prove a new rigidity theorem on smooth closed Weingarten hypersurfaces in hyperbolic space, and a new class of Alexandrov-Fenchel type inequalities for smooth horospherically convex hypersurfaces in hyperbolic space.
\end{abstract}

\maketitle
\tableofcontents

\section{Introduction}
Let $X_0: M^n\to \mathbb{H}^{n+1}$ be a smooth embedding such that $M_0=X_0(M)$ is a closed smooth hypersurface in the hyperbolic space $\mathbb{H}^{n+1}$. We consider a smooth family of immersions $X:M^n\times [0,T)\rightarrow \mathbb{H}^{n+1}$ satisfying
\begin{equation}\label{flow-VMCF-0}
 \left\{\begin{aligned}
 \frac{\partial}{\partial t}X(x,t)=&~(\phi(t)-\Psi(x,t))\nu(x,t),\\
 X(\cdot,0)=&~X_0(\cdot),
  \end{aligned}\right.
 \end{equation}
 where $\nu(x,t)$ is the unit outward normal of $M_t=X(M,t)$, $\Psi$ is a smooth curvature function evaluated at the point $(x,t)$ of $M_t$, the global term $\phi(t)$ is chosen to impose a constraint on the enclosed volume or quermassintegrals of $M_t$.

The volume preserving mean curvature flow in hyperbolic space was first studied by Cabezas-Rivas and Miquel \cite{Cab-Miq2007} in 2007. By imposing horospherically convexity (the condition that all principal curvatures exceed $1$, which will also be called h-convex) on the initial hypersurface, they proved that the solution exists for all time and converges smoothly to a geodesic sphere.  Some other mixed volume preserving flows were considered in \cite{Mak2012,WX} with speed given by homogeneous degree one functions of the principal curvatures.
Recently Bertini and Pipoli \cite{Be-Pip2016} succeeded in treating flows by more general functions of mean curvature, including in particular any positive power of mean curvature. In a recent paper \cite{And-Wei2017-2}, the first and the third authors proved the smooth convergence of quermassintegral preserving flows with speed given by any positive power of a homogeneous degree one function $f$ of the principal curvatures for which the dual function $f_*(x_1,\cdots,x_n) = (f(x_1^{-1},\cdots,x_n^{-1}))^{-1}$ is concave and approaches zero on the boundary of the positive cone. This includes in particular the volume preserving flow by positive powers of $k$-th mean curvature for h-convex hypersurfaces. Note that in all the above mentioned work, the initial hypersurface is assumed to be h-convex.

One reason to consider constrained flows of the kind considered here is to prove geometric inequalities:  In particular, the convergence of the volume-preserving mean curvature flow to a sphere implies that the area of the initial hypersurface is no less than that of a geodesic sphere with the same enclosed volume, since the area is non-increasing while the volume remains constant under the flow.  The same motivation lies behind  \cite{WX}, where inequalities between quermassintegrals were deduced from the convergence of certain flows.

In this paper, we make the following contributions:
\begin{itemize}
  \item[(1)] In the first part of the paper, we weaken the horospherical convexity condition, allowing instead hypersurfaces for which the intrinsic sectional curvatures are positive.  We
  consider the flow \eqref{flow-VMCF-0} for hypersurfaces with positive sectional curvature and with speed $\Psi$ given by any positive power of a smooth, symmetric, strictly increasing and homogeneous of degree one function of the Weingarten matrix $\mathcal{W}$ of $M_t$. Here we say a hypersurface $M$ in hyperbolic space has positive sectional curvature if its sectional curvature $R_{ijij}^M>0$ for any $1\leq i< j\leq n$, which by Gauss equation is equivalent to the principal curvatures of $M$ satisfying $\kappa_i\kappa_j>1$ for $1\leq i\neq j\leq n$.  This is a weaker condition than h-convexity.  As a consequence we deduce inequalities between volume and other quermassintegrals for hypersurfaces with positive sectional curvature, extending inequalities previously known only for horospherically convex hypersurfaces.
  \item[(2)] In the second part of this paper, we consider flows \eqref{flow-VMCF-0} for strictly h-convex hypersurfaces in which the speed $\Psi$ is homogeneous as a function of the shifted Weingarten matrix $\mathcal{W}-\mathrm{I}$ of $M_t$, rather than the Weingarten matrix itself.  Using these flows we are able to prove a new class of integral inequalities for horospherically convex hypersurfaces.
  \item[(3)] In order to understand these new functionals we introduce some new machinery for horospherically convex regions, including a horospherical Gauss map and a horospherical support function. We also develop an interesting connection (closely related to the results of \cite{EGM}) between flows of h-convex hypersurfaces in hyperbolic space by functions of principal curvatures, and conformal flows of conformally flat metrics on $S^n$ by functions of the eigenvalues of the Schouten tensor.  This allows us to translate our results to convergence theorems for metric flows, and our isoperimetric inequalities to corresponding results for conformally flat metrics. We expect that these will prove useful in future work.
\end{itemize}
We will describe our results in more detail in the rest of this section:

\subsection{Volume preserving flow with positive sectional curvature}
Suppose that the initial hypersurface $M_0$ has positive sectional curvature.    We consider the smooth family of immersions $X:M^n\times [0,T)\rightarrow \mathbb{H}^{n+1}$ satisfying
\begin{equation}\label{flow-VMCF}
 \left\{\begin{aligned}
 \frac{\partial}{\partial t}X(x,t)=&~(\phi(t)-F^{\alpha}(\mathcal{W}))\nu(x,t),\\
 X(\cdot,0)=&~X_0(\cdot),
  \end{aligned}\right.
 \end{equation}
 where $\alpha>0$, $\nu(x,t)$ is the unit outward normal of $M_t=X(M,t)$, $F$ is a smooth, symmetric, strictly increasing and homogeneous of degree one function of the Weingarten matrix $\mathcal{W}$ of $M_t$. The global term $\phi(t)$ in \eqref{flow-VMCF} is defined by
\begin{equation}\label{s1:phit}
  \phi(t)=\frac{1}{|M_t|}\int_{M_t}F^{\alpha} d\mu_t
\end{equation}
such that the volume of $\Omega_t$ remains constant along the flow \eqref{flow-VMCF}, where $d\mu_t$ is the area measure on $M_t$ with respect to the induced metric.

Since $F(\mathcal{W})$ is symmetric with respect to the components of $\mathcal{W}$, by a theorem of Schwarz \cite{Scharz75} we can write $F(\mathcal{W})=f(\kappa)$ as a symmetric function of the eigenvalues of $\mathcal{W}$. We assume that $f$ satisfies the following assumption:
\begin{assump}\label{s1:Asum}
Suppose $f$ is a smooth symmetric function defined on the positive cone $\Gamma_+:=\{\kappa=(\kappa_1,\cdots,\kappa_n)\in \mathbb{R}^n: \kappa_i>0,~\forall~i=1,\cdots,n\}$, and satisfies
\begin{itemize}
  \item[(i)]  $f$ is positive, strictly increasing, homogeneous of degree one and is normalized such that $f(1,\cdots,1)=1$;
  \item[(ii)] For any $i\neq j$,
  \begin{equation}\label{s1:asum-1}
    (\frac{\partial f}{\partial \kappa_i}\kappa_i-\frac{\partial f}{\partial \kappa_j}\kappa_j)(\kappa_i-\kappa_j)~\geq~0.
  \end{equation}
  \item[(iii)] For all $(y_1,\cdots,y_n)\in \mathbb{R}^n$,
  \begin{equation}\label{s1:asum-2}
    \sum_{i,j}\frac{\partial^2 \log f}{\partial\kappa_i\partial\kappa_j}y_iy_j+\sum_{i=1}^n\frac 1{\kappa_i}\frac{\partial\log f}{\partial \kappa_i}y_i^2~\geq~0.
  \end{equation}
\end{itemize}
\end{assump}
\noindent Examples satisfying Assumption \ref{s1:Asum} include  $f=n^{-1/k}S_k^{1/k}$ $(k>0)$ and $f=E_k^{1/k}$ (see, e.g., \cite{GaoLM17,GuanMa03}), where
\begin{equation*}
  E_{k}={\binom{n}{k}}^{-1}\sigma_k(\kappa)={\binom{n}{k}}^{-1}\sum_{1\leq i_1<\cdots<i_k\leq n}\kappa_{i_1}\cdots \kappa_{i_k},\qquad k=1,\cdots,n.
\end{equation*}
is the (normalized) $k$-th mean curvature of $M_t$ and $ S_k(\kappa)=\sum_{i=1}^n\kappa_i^k$ is the $k$-th power sum of $\kappa$ for $k>0$.  The inequalities \eqref{s1:asum-1} and \eqref{s1:asum-2} are equivalent to the statement that $\log F$ is a convex function of the components of $\log{\mathcal W}$, which is the map with the same eigenvectors as ${\mathcal W}$ and eigenvalues $\log\kappa_i$.  In particular, if $f_1$ and $f_2$ are two symmetric functions satisfying \eqref{s1:asum-1} and \eqref{s1:asum-2}, then the function $f_1^\alpha$ with $\alpha>0$ and the product $f_1f_2$ also satisfy \eqref{s1:asum-1} and \eqref{s1:asum-2}. Note that the Cauchy-Schwarz inequality and \eqref{s1:asum-2} imply that any symmetric function $f$ satisfying \eqref{s1:asum-2} must be inverse concave, i.e., its dual function
 $$f_*(z_1,\cdots,z_n)=f(z_1^{-1},\cdots,z_n^{-1})^{-1}$$
is concave with respect to its argument.

The first result of this paper is the following convergence result for the flow \eqref{flow-VMCF}:
\begin{thm}\label{thm1-1}
Let $X_0: M^n\to \mathbb{H}^{n+1}$ be a smooth embedding such that $M_0=X_0(M)$ is a closed hypersurface in $\mathbb{H}^{n+1}$ $(n\geq 2)$ with positive sectional curvature. Assume that $f$ satisfies Assumption \ref{s1:Asum}, and either
\begin{itemize}
  \item[(i)] $f_*$ vanishes on the boundary of $\Gamma_+$, and
  \begin{equation}\label{thm1-cond1}
    \lim_{x\to 0+}f(x,\frac 1x,\cdots,\frac 1x)~=~+\infty,
  \end{equation}
and $\alpha>0$, or
 \item[(ii)] $n=2$, $f=(\kappa_1\kappa_2)^{1/2}$ and $\alpha\in [1/2,2]$.
\end{itemize}
Then the flow \eqref{flow-VMCF} with global term $\phi(t)$ given by \eqref{s1:phit} has a smooth solution $M_t$ for all time $t\in [0,\infty)$, and $M_t$ has positive sectional curvature for each $t>0$ and converges smoothly and exponentially to a geodesic sphere of radius $r_{\infty}$ determined by $\mathrm{Vol}(B(r_{\infty}))=\mathrm{Vol}(\Omega_0)$  as $t\to\infty$.
\end{thm}

\begin{rem}
Examples of function $f$ satisfying Assumption \ref{s1:Asum} and the condition (i) of Theorem \ref{thm1-1} include:
\begin{itemize}
  \item[a).] $n\geq 2$, $f=n^{-1/k}S_k^{1/k}$ with $k>0$;
  \item[b).] $n\geq 3$, $f= E_k^{1/k}$ with  $k=1,\cdots,n$;
  \item[c).] $n=2, f=(\kappa_1+\kappa_2)/2$.
\end{itemize}\end{rem}
\begin{rem}
We remark that the contracting curvature flows for surfaces with positive scalar curvature in hyperbolic 3-space $\mathbb{H}^3$ have been studied by the first two authors in a recent work \cite{And-chen2014}.
\end{rem}

As a key step in the proof of Theorem \ref{thm1-1}, we prove in \S \ref{sec:PSC} that the positivity of sectional curvatures of the evolving hypersurface $M_t$ is preserved along the flow \eqref{flow-VMCF} with any $f$ satisfying Assumption \ref{s1:Asum} and any $\alpha>0$. In order to show that the positivity of sectional curvatures are preserved, we consider the sectional curvature as a function on the frame bundle $O(M)$ over $M$, and apply a maximum principle.  This requires a rather delicate computation, using inequalities on the Hessian on the total space of $O(M)$ to show the required inequality on the time derivative at a minimum point.  The argument is related to that used by the first author to prove a generalised tensor maximum principle in \cite{And2007}*{Theorem 3.2}, but cannot be deduced directly from that result.  The argument combines the ideas of the generalised tensor maximum principle with those of vector bundle maximum principles for reaction-diffusion equations \cites{AH,Ha1986}.

 We remark that the flow \eqref{flow-VMCF} with
\begin{equation}\label{s1:f-quoti}
  f~=~\left(\frac{E_k}{E_l}\right)^{\frac 1{k-l}},\quad 1\leq l<k\leq n
\end{equation}
and any power $\alpha>0$ does not preserve positive sectional curvatures:  Counterexamples can be produced in the spirit of the constructions in \cite{Andrews-McCoy-Zheng}*{Sections 4--5}.

The remaining parts of the proof of Theorem \ref{thm1-1} will be given in \S \ref{sec:thm1-pf}. In \S \ref{sec:4-1}, we will derive a uniform estimate on the inner radius and outer radius of the evolving domains $\Omega_t$ along the flow \eqref{flow-VMCF}. Recall that the inner radius $\rho_-$ and outer radius $\rho_+$ of a bounded domain $\Omega$ are defined as
\begin{equation*}
  \rho_-= ~\sup \bigcup_{p\in\Omega}\{\rho>0:\ B_{\rho}(p)\subset\Omega\} \qquad
  \rho_+= ~\inf\bigcup_{p\in\Omega}\{\rho>0: \Omega\subset B_{\rho}(p)\},
\end{equation*}
where $B_{\rho}(p)$ denotes the geodesic ball of radius $\rho$ and centered at some point $p$ in the hyperbolic space. All the previous papers \cite{And-Wei2017-2,Be-Pip2016,Cab-Miq2007,Mak2012,WX} on constrained curvature flows in hyperbolic space focus on horospherically convex domains, which have the property that $\rho_+\leq c(\rho_-+\rho_-^{1/2})$, see e.g. \cite{Cab-Miq2007,Mak2012}. However, no such property is known for hypersurfaces with positive sectional curvature. Our idea to overcome this obstacle is to use an Alexandrov reflection argument to bound the diameter of the domain $\Omega_t$ enclosed by the flow hypersurface $M_t$. Then we project the domain $\Omega_t$ to the unit ball in Euclidean space $\mathbb{R}^{n+1}$ via the Klein model of the hyperbolic space. The upper bound on the diameter of $\Omega_t$ implies that this map has bounded distortion. This together with the preservation of the volume of $\Omega_t$ gives a uniform lower bound on the inner radius of $\Omega_t$.

Then in \S \ref{sec:4-2} we adapt Tso's technique \cite{Tso85}  to derive an upper bound on the speed if $f$ satisfies Assumption \ref{s1:Asum}, where the positivity of sectional curvatures of $M_t$ will be used to estimate the zero order terms of the evolution equation of the auxiliary function. In \S \ref{sec:4-3}, we will complete the proof of Theorem \ref{thm1-1} by obtaining uniform bounds on the principal curvatures. In the case (i) of Theorem \ref{thm1-1}, the upper bound of $f$ together with the positivity of sectional curvatures imply the uniform two-sides positive bound of the principal curvatures of $M_t$. In the case (ii) of Theorem \ref{thm1-1}, the estimate $1\leq \kappa_1\kappa_1=f(\kappa)^2\leq C$ does not prevent $\kappa_2$ from going to infinity. Instead, we will obtain the estimate on the pinching ratio $\kappa_2/\kappa_1$ by applying the maximum principle to the evolution equation of
$G(\kappa_1,\kappa_2)=(\kappa_1\kappa_2)^{\alpha-2}(\kappa_2-\kappa_1)^2$ with $\alpha\in [1/2,2]$. This idea has been applied by the first two authors in \cite{And1999,And-chen2012} to prove the pinching estimate for the contracting flow by powers of Gauss curvature in $\mathbb{R}^3$.  Once we have the uniform estimate on the principal curvatures of the evolving hypersurfaces, higher regularity estimates can be derived by a standard argument.  A continuation argument then yields the long time existence of the flow and the Alexandrov reflection argument as in \cite[\S 6]{And-Wei2017-2} implies the smooth convergence of the flow to a geodesic sphere.

\subsection{Alexandrov-Fenchel inequalities}
The volume preserving curvature flow is a useful tool in the study of hypersurface geometry. We will illustrate an application of Theorem \ref{thm1-1} in the proof of Alexandrov-Fenchel type inequalities (involving the quermassintegrals) for hypersurfaces in hyperbolic space. Recall that for a convex domain $\Omega$ in hyperbolic space, the quermassintegral $W_k(\Omega)$ is defined as follows (see \cite{Sant2004,Sol2006}):\footnote{Note that the definition \eqref{S1:Wk-def1} is different with the definition given in \cite{Sol2006} by a constant multiple $\frac{n+1-k}{n+1}$.}
\begin{equation}\label{S1:Wk-def1}
  W_k(\Omega)=~\frac{\omega_{k-1}\cdots\omega_0}{\omega_{n-1}\cdots\omega_{n-k}}\int_{\mathcal{L}_k}\chi(L_k\cap\Omega)dL_k,\quad k=1,\cdots,n,
\end{equation}
where $\mathcal{L}_k$  is the space of $k$-dimensional totally geodesic subspaces $L_k$ in $\mathbb{H}^{n+1}$ and $\omega_n$ denotes the area of $n$-dimensional unit sphere in Euclidean space. The function $\chi$ is defined to be $1$ if  $L_k\cap\Omega\neq \emptyset$ and to be $0$ otherwise. Furthermore, we set
\begin{equation*}
  W_0(\Omega)=|\Omega|,\quad W_{n+1}(\Omega)=|\mathbb{B}^{n+1}|=\frac{\omega_n}{n+1}.
\end{equation*}
If the boundary of $\Omega$ is smooth, we can define the principal curvatures $\kappa=(\kappa_1,\cdots,\kappa_n)$ and the curvature integrals
\begin{equation}\label{s1:CurInt}
  V_{n-k}(\Omega)~=~\int_{\partial\Omega}E_k(\kappa)d\mu,\quad k=0,1,\cdots,n
\end{equation}
of the boundary $M=\partial\Omega$. The quermassintegrals and the curvature integrals of a smooth convex domain $\Omega$ in $\mathbb{H}^{n+1}$ are related by the following equations (see \cite{Sol2006}):
\begin{align}
  V_{n-k}(\Omega) ~=&~ (n-k)W_{k+1}(\Omega)+ kW_{k-1}(\Omega),\quad k=1,\cdots,n\label{s1:quermass-1}\\
  V_n(\Omega) ~=&~nW_1(\Omega)~=~|\partial\Omega|.\label{s1:quermass-2}
\end{align}
In \cite{WX}, Wang and Xia proved the Alexandrov-Fenchel inequalities for smooth h-convex domain $\Omega$ in $\mathbb{H}^{n+1}$, which states that there holds
\begin{equation}\label{s1:AF-WX}
  W_k(\Omega)~\geq f_k\circ f_l^{-1}(W_l(\Omega))
\end{equation}
for any $0\leq l<k\leq n$, with equality if and only if $\Omega$ is a geodesic ball, where $f_k: \mathbb{R}_+\to \mathbb{R}_+$ is an increasing function  defined by $f_k(r)=W_k(B(r))$, the $k$-th quermassintegral of the geodesic ball of radius $r$. The proof in \cite{WX} is by applying the quermassintegral preserving flow for smooth h-convex hypersurfaces with speed given by the quotient \eqref{s1:f-quoti} and $\alpha=1$, and is similar with the Euclidean analogue considered by McCoy \cite{McC2005}.  The inequality \eqref{s1:AF-WX} can imply the following inequality
\begin{equation}\label{s1:AF-WX2}
  \int_{\partial\Omega}E_kd\mu~\geq~|\partial\Omega|\left(1+\left(\frac{|\partial\Omega|}{\omega_n}\right)^{-2/n}\right)^{k/2}
\end{equation}
for smooth h-convex domains, which compares the curvature integral \eqref{s1:CurInt} and the boundary area. Note that the inequality \eqref{s1:AF-WX2} with $k=2$ was proved earlier by the third author with Li and Xiong \cite{LWX-2014} for star-shaped and $2$-convex domains using the inverse curvature flow in hyperbolic space.  For the other even $k$, the inequality \eqref{s1:AF-WX2} was also proved for smooth h-convex domains using the inverse curvature flow by Ge, Wang and Wu \cite{GWW-2014JDG}. It's an interesting problem to prove the inequalities \eqref{s1:AF-WX} and \eqref{s1:AF-WX2} under an assumption that is  weaker than h-convexity.

Applying the result in Theorem \ref{thm1-1}, we show that the \emph{h-convexity} assumption for the inequality \eqref{s1:AF-WX} can be replaced by the weaker assumption of \emph{positive sectional curvature} in the case $l=0$ and $1\leq k\leq n$.
\begin{cor}\label{s1:cor-1}
Let $M=\partial\Omega$ be a smooth closed hypersurface in $\mathbb{H}^{n+1}$ which has positive sectional curvature and encloses a smooth bounded domain $\Omega$. Then for any $n\geq 2$ and $k=1,\cdots,n$, we have
\begin{equation}\label{s1:AF-0}
  W_k(\Omega)~\geq f_k\circ f_0^{-1}(W_0(\Omega)),
\end{equation}
where $f_k: \mathbb{R}_+\to \mathbb{R}_+$ is an increasing function  defined by $f_k(r)=W_k(B(r))$, the $k$-th quermassintegral of the geodesic ball of radius $r$. Moreover, equality holds in \eqref{s1:AF-0} if and only if $\Omega$ is a geodesic ball.
\end{cor}
The quermassintegral $W_k(\Omega_t)$ of the evolving domain $\Omega_t$ along the flow \eqref{flow-VMCF} with $F=E_k^{1/k}$ satisfies (see Lemma \ref{s2:lem2})
\begin{equation*}
  \frac d{dt}W_k(\Omega_t)~=~\int_{M_t}E_k(\phi(t)-E_k^{\alpha/k})d\mu_t,
\end{equation*}
which is non-positive for each $\alpha>0$ by the choice \eqref{s1:phit} of $\phi(t)$ and the H\"{o}lder inequality. This means that $W_k(\Omega_t)$ is monotone decreasing along the flow \eqref{flow-VMCF} with $F=E_k^{1/k}$ unless $E_k$ is constant on $M_t$ (which is equivalent to $M_t$ being a geodesic sphere). Then Corollary \ref{s1:cor-1} follows from the monotonicity of $W_k$ and the convergence result in Theorem \ref{thm1-1}.

\subsection{Volume preserving flow for horospherically convex hypersurfaces}

In the second part of this paper, we will consider the flow of h-convex hypersurfaces in hyperbolic space with speed given by functions of the shifted Weingarten matrix $\mathcal{W}-\mathrm{I}$ plus a global term chosen to preserve modified quermassintegrals of the evolving domains.
Let us first define the following modified quermassintegrals:
 \begin{equation}\label{s1:Wk-td}
  \widetilde{W}_k(\Omega)~:=~\sum_{i=0}^k(-1)^{k-i}\binom{k}{i}W_i(\Omega),\quad k=0,\cdots,n,
 \end{equation}
for a h-convex domain $\Omega$ in hyperbolic space.  Thus $\widetilde{W}_k$ is a linear combination of the quermassintegrals of $\Omega$. In particular, $\widetilde{W}_0(\Omega)=|\Omega|$ is the volume of $\Omega$. The modified quermassintegrals defined in \eqref{s1:Wk-td} satisfy the following property:
\begin{prop}\label{s1:wk-prop}
The modified quermassintegral ${\widetilde W}_k$ is monotone with respect to inclusion for h-convex domains:  That is, if $\Omega_0$ and $\Omega_1$ are h-convex domains with $\Omega_0\subset\Omega_1$, then $\widetilde W_k(\Omega_0)\leq\widetilde W_k(\Omega_1)$.
\end{prop}
This property is not obvious from the definition \eqref{s1:Wk-td} and its proof will be given in \S \ref{sec:h-convex}. We will first investigate some of the properties of horospherically convex regions in hyperbolic space $\mathbb{H}^{n+1}$. In particular, for such regions we define a \emph{horospherical Gauss map}, which
is a map to the unit sphere, and we show that each horospherically convex region is completely described in terms of a scalar function $u$ on the sphere $\mathbb{S}^n$ which we call the \emph{horospherical
support function}. There are interesting formal similarities between this situation and that
of convex Euclidean bodies. We show that the h-convexity of a region $\Omega$ is equivalent to that the following matrix
\begin{equation*}
A_{ij} = \bar\nabla_j\bar\nabla_k\varphi-\frac{|\bar\nabla\varphi|^2}{2\varphi}\bar g_{ij}+\frac{\varphi-\varphi^{-1}}{2}\bar g_{ij}
\end{equation*}
on the sphere $\mathbb{S}^n$ is positive definite, where $\bar{g}_{ij}$ is the standard round metric on $\mathbb{S}^n$, $\varphi=\E^u$ and $u$ is the horospherical support function of $\Omega$. The shifted Weingarten matrix $\mathcal{W}-\mathrm{I}$ is related to the matrix $A_{ij}$ by 
\begin{equation}\label{eq:A-vs-W}
A_{ij} = \varphi^{-1}\left[\left({\mathcal W}-\mathrm{I}\right)^{-1}\right]_i^k\bar g_{kj}.
\end{equation}
Using this characterization of h-convex domains, for any two h-convex domains $\Omega_0$ and $\Omega_1$ with $\Omega_0\subset\Omega_1$ we can find a foliation of h-convex domains $\Omega_t$ which is expanding from $\Omega_0$ to $\Omega_1$. This can be used to prove Proposition \ref{s1:wk-prop} by computing the variation of $\widetilde{W}_k$. We expect that the description of horospherically convex regions which we develop here will be useful in further investigations beyond the scope of this paper.

The flow we will consider is the following:
\begin{equation}\label{flow-VMCF-2}
 \left\{\begin{aligned}
 \frac{\partial}{\partial t}X(x,t)=&~(\phi(t)-F(\mathcal{W}-\mathrm{I}))\nu(x,t),\\
 X(\cdot,0)=&~X_0(\cdot)
  \end{aligned}\right.
 \end{equation}
for smooth and strictly h-convex hypersurface in hyperbolic space, where $F$ is a smooth, symmetric, homegeneous of degree one function of the shifted Weingarten matrix $\mathcal{W}-\mathrm{I}=(h_i^j-\delta_i^j)$. For simplicity, we denote $S_{ij}=h_i^j-\delta_i^j$. Note that the eigenvalues of $(S_{ij})$ are the shifted principal curvatures $\lambda=(\lambda_1,\cdots,\lambda_n)=(\kappa_1-1,\cdots,\kappa_n-1)$. Again by a theorem of Schwarz \cite{Scharz75}, $F(\mathcal{W}-\mathrm{I})=f(\lambda)$, where $f$ is a smooth symmetric function of $n$ variables $\lambda=(\lambda_1,\cdots,\lambda_n)$. We choose the global term $\phi(t)$ in \eqref{flow-VMCF-2} as
 \begin{equation}\label{s1:phit-2}
  \phi(t)=\left(\int_{M_t}E_l(\lambda)d\mu_t\right)^{-1}\int_{M_t}E_l(\lambda)Fd\mu_t,\quad l=0,\cdots,n
\end{equation}
such that $\widetilde{W}_l(\Omega_t)$ remains constant, where $\Omega_t$ is the domain enclosed by the evolving hypersurface $M_t$.

We will prove the following result for the flow \eqref{flow-VMCF-2} with $\phi(t)$ given in \eqref{s1:phit-2}.
\begin{thm}\label{thm1-5}
Let $n\geq 2$ and $X_0: M^n\to \mathbb{H}^{n+1}$ be a smooth embedding such that $M_0=X_0(M)$ is a smooth closed and strictly h-convex hypersurface in $\mathbb{H}^{n+1}$. If $f$ is a smooth, symmetric, increasing and homogeneous of degree one function, and either
\begin{itemize}
  \item[(i)] $f$ is concave and f approaches zero on the boundary of the positive cone $\Gamma_+$, or
  \item[(ii)] $f$ is concave and inverse concave, or
  \item[(iii)] $f$ is inverse concave and its dual function $f_*$ approaches zero on the boundary of positive cone $\Gamma_+$, or
  \item[(iv)] $n=2$,
\end{itemize}
then the flow \eqref{flow-VMCF-2} with the global term $\phi(t)$ given by \eqref{s1:phit-2} has a smooth solution $M_t$ for all time $t\in [0,\infty)$, and $M_t$ is strictly h-convex for any $t>0$ and converges smoothly and exponentially to a geodesic sphere of radius $r_{\infty}$ determined by $\widetilde{W}_l(B(r_{\infty}))=\widetilde{W}_l(\Omega_0)$  as $t\to\infty$.
\end{thm}

Constrained curvature flows in hyperbolic space by homogeneous of degree one, concave and inverse concave function of the principal curvatures were studied by Makowski \cite{Mak2012} and Wang and Xia in \cite{WX}. The quermassintegral preserving flow by any positive power of a homogeneous of degree one function of the principal curvatures, which is inverse concave and its dual function $f_*$ approaches zero on the boundary of positive cone $\Gamma_+$, was studied recently by the first and the third authors in \cite{And-Wei2017-2}. Note that the speed function $f$ of the flow \eqref{flow-VMCF-2} in Theorem \ref{thm1-5} is not a homogeneous function of the principal curvatures $\kappa_i$ and there are essential differences in the analysis compared with the previously mentioned work \cite{And-Wei2017-2,Mak2012,WX}.

The key step in the proof of Theorem \ref{thm1-5} is a pinching estimate for the shifted principal curvatures $\lambda_i$. That is, we will show that the ratio of the largest shifted principal curvature $\lambda_n$ to the smallest shifted principal curvature $\lambda_1$ is controlled by its initial value along the flow \eqref{flow-VMCF-2}. For the proof, we adapt methods from the proof of pinching estimates of the principal curvatures for contracting curvature flows \cite{And1994,And2007,And2010,Andrews-McCoy-Zheng} and the constrained curvature flows in Euclidean space  \cite{McC2005,Mcc2017}.  In particular, in the case (iii) we define the tensor $T_{ij}=S_{ij}-\varepsilon F\delta_i^j$ and show that the positivity of $T_{ij}$ is preserved by applying the tensor maximum principle (proved by the first author in \cite{And2007}). The inverse concavity is used to estimate the sign of the gradient terms. This case is similar with the pinching estimate for the contracting curvature flow in Euclidean case in \cite[Lemma 11]{Andrews-McCoy-Zheng}. Although the proof there is given in terms of the Gauss map parametrisation of the convex solutions of the flow in Euclidean space, which is not available in hyperbolic space, we can deal with the gradient terms directly using the inverse concavity property of $f$.

To prove Theorem \ref{thm1-5},  we next show that the inner radius and outer radius of the enclosed domain $\Omega_t$ of the evolving hypersurface $M_t$ satisfies a uniform estimate $0<C^{-1}<\rho_-(t)\leq \rho_+(t)\leq C$ for some positive constant $C$. This relies on the preservation of $\widetilde{W}_l(\Omega_t)$ and the monotonicity of $\widetilde{W}_l$ under the inclusion of h-convex domains stated in Proposition \ref{s1:wk-prop}. With the estimate on the inner radius and outer radius, the technique of Tso \cite{Tso85} yields the upper bound on $F$ and the Harnack inequality of Krylov and Safonov \cite{KS81} yields the lower bound on $F$. The pinching estimate then gives the estimate on the shifted principal curvatures $\lambda_i$. The long time existence and the convergence of the flow then follows by a standard argument.

The result in Theorem \ref{thm1-5} is useful in the study of the geometry of hypersurfaces. The first application of Theorem \ref{thm1-5} is the following rigidity result.
\begin{cor}\label{s1:cor-2}
Let $M$ be a smooth, closed and strictly h-convex hypersurface in $\mathbb{H}^{n+1}$ with principal curvatures $\kappa=(\kappa_1,\cdots,\kappa_n)$ satisfying $f(\lambda)=C$
for some constant $C>0$, where $\lambda=(\lambda_1,\cdots,\lambda_i)$ with $\lambda_i=\kappa_i-1$ and $f$ is a symmetric function satisfying the condition in Theorem \ref{thm1-5}. Then $M$ is a geodesic sphere.
\end{cor}

The second application of Theorem \ref{thm1-5} is a new class of Alexandrov-Fenchel type inequalities between quermassintegrals of h-convex hypersurface in hyperbolic space.
\begin{cor}\label{s1:cor-3}
Let $M=\partial\Omega$ be a smooth, closed and strictly h-convex hypersurface in $\mathbb{H}^{n+1}$. Then for any $0\leq l<k\leq n$, there holds
\begin{equation}\label{s1:AF}
 \widetilde{W}_k(\Omega)~\geq~ \tilde{f}_k\circ \tilde{f}_l^{-1}(\widetilde{W}_l(\Omega)),
\end{equation}
with equality holding if and only if $\Omega$ is a geodesic ball.  Here the function $\tilde{f}_k: \mathbb{R}_+\to \mathbb{R}_+$ is defined by $\tilde{f}_k(r)=\widetilde{W}_k(B(r))$, which is an  increasing function by Proposition \ref{s1:wk-prop}. $\tilde{f}_l^{-1}$ is the inverse function of $\tilde{f}_l$.
\end{cor}

The inequality \eqref{s1:AF} can be obtained by applying Theorem \ref{thm1-5} with $f$ chosen as
\begin{equation}\label{s1:F}
  f=\left(\frac{E_k(\lambda)}{E_l(\lambda)}\right)^{\frac 1{k-l}},\quad 0\leq l<k\leq n
\end{equation}
in the flow \eqref{flow-VMCF-2}. We see that along the flow \eqref{flow-VMCF-2} with such $f$, the modified quermassintegral $\widetilde{W}_l(\Omega_t)$ remains a constant and $\widetilde{W}_k(\Omega_t)$ is monotone decreasing in time by the H\"{o}lder inequality.  In fact, by Lemma \ref{s2:lem3} the modified quermassintegral evolves by
\begin{align}\label{s1:eq2}
  \frac d{dt}\widetilde{W}_k(\Omega_t)~=&~ \int_{M_t}E_k(\lambda)\left(\phi(t)-\left(\frac{E_k(\lambda)}{E_l(\lambda)}\right)^{\frac 1{k-l}}\right)d\mu_t.
\end{align}
Applying the H\"{o}lder inequality to the equation \eqref{s1:eq2} yields that $\widetilde{W}_k(\Omega_t)$ is monotone decreasing in time unless $E_k(\lambda)=C E_l(\lambda)$ on $M_t$ (which is equivalent to that $M_t$ is a geodesic sphere by Corollary \ref{s1:cor-2}). Since the flow exists for all time and converges to a geodesic sphere $B_r$, the inequality \eqref{s1:AF} follows from the monotonicity of $\widetilde{W}_k(\Omega_t)$ and the preservation of $\widetilde{W}_l(\Omega_t)$.

\begin{rem}
We remark that the inequalities \eqref{s1:AF} are new and can be viewed as an improvement of the inequalities \eqref{s1:AF-WX}. For example, the inequality \eqref{s1:AF} with $l=0$ implies that
\begin{equation}\label{s1:AF-rem1}
  \sum_{i=0}^k(-1)^{k-i}\binom ki\biggl(W_i(\Omega)-f_i\circ f_0^{-1}(W_0(\Omega))\biggr)~\geq~0.
\end{equation}
By induction on $k$, \eqref{s1:AF-rem1} implies that each $W_i(\Omega)-f_i\circ f_0^{-1}(W_0(\Omega))$ is nonnegative for $h$-convex domains. Thus our inequalities \eqref{s1:AF} imply that the linear combinations of $W_i(\Omega)-f_i\circ f_0^{-1}(W_0(\Omega))$ as in \eqref{s1:AF-rem1} are also nonnegative for h-convex domains.
\end{rem}

\begin{ack}
The second author is grateful to the Mathematical Sciences Institute at the Australian National University for its hospitality during his visit, when some of this work was completed.
\end{ack}

\section{Preliminaries}\label{sec:pre}
In this section we collect some properties of smooth symmetric functions $f$ of $n$ variables, and recall the evolution equations of geometric quantities along the flows \eqref{flow-VMCF} and \eqref{flow-VMCF-2}.

\subsection{Properties of symmetric functions}
For a smooth symmetric function $F(A)=f(\kappa(A))$, where $A=(A_{ij})\in \mathrm{Sym}(n)$ is symmetric matrix and $\kappa(A)=(\kappa_1,\cdots,\kappa_n)$ give the eigenvalues of $A$, we denote by $\dot{F}^{ij}$ and $\ddot{F}^{ij,kl}$ the first and second derivatives of $F$ with respect to the components of its argument, so that
\begin{equation*}
  \frac{\partial}{\partial s}F(A+sB)\bigg|_{s=0}=\dot{F}^{ij}(A)B_{ij}
\end{equation*}
and
\begin{equation*}
  \frac{\partial^2}{\partial s^2}F(A+sB)\bigg|_{s=0}=\ddot{F}^{ij,kl}(A)B_{ij}B_{kl}
\end{equation*}
for any two symmetric matrixs $A,B$.  We also use the notation
\begin{equation*}
  \dot{f}^i(\kappa)=\frac{\partial f}{\partial \kappa_i}(\kappa),\quad  \ddot{f}^{ij}(\kappa)=\frac{\partial^2 f}{\partial \kappa_i\partial \kappa_j}(\kappa).
\end{equation*}
for the first and the second derivatives of $f$ with respect to $\kappa$. At any diagonal $A$ with distinct eigenvalues $\kappa=\kappa(A)$, the first derivatives of $F$ satisfy
\begin{equation*}
  \dot{F}^{ij}(A)~=~\dot{f}^i(\kappa)\delta_{ij}
\end{equation*}
and the second derivative of $F$ in direction $B\in \mathrm{Sym}(n)$ is given in terms of $\dot{f}$ and $\ddot{f}$ by  (see \cite{And2007}):
\begin{equation}\label{s2:F-ddt}
  \ddot{F}^{ij,kl}(A)B_{ij}B_{kl}=\sum_{i,j}\ddot{f}^{ij}(\kappa)B_{ii}B_{jj}+2\sum_{i>j}\frac{\dot{f}^i(\kappa)-\dot{f}^j(\kappa)}{\kappa_i-\kappa_j}B_{ij}^2.
\end{equation}
This formula makes sense as a limit in the case of any repeated values of $\kappa_i$.

From the equation \eqref{s2:F-ddt}, we have
\begin{lem}\label{s2:lem0}
Suppose $A$ has distinct eigenvalues $\kappa=\kappa(A)$.  Then $F$ is concave at $A$ if and only if $f$ is concave at $\kappa$ and
\begin{equation}\label{s2:f-conc}
  \left(\dot{f}^k-\dot{f}^l\right)(\kappa_k-\kappa_l)~\leq ~0,\quad \forall~k\neq l.
\end{equation}
\end{lem}

In this paper, we also need the inverse concavity property of $f$ in many cases. We include the properties of inverse concave function in the following lemma.
 \begin{lem}[\cite{And2007,And-Wei2017-2}]\label{s2:lem1}
\begin{itemize}
 \item[(i)]If $f$ is inverse concave, then
 \begin{equation}\label{s2:f-invcon-1}
   \sum_{k,l=1}^n\ddot{f}^{kl}y_ky_l+2\sum_{k=1}^n\frac {\dot{f}^k}{\kappa_k}y_k^2~\geq ~2f^{-1}(\sum_{k=1}^n\dot{f}^ky_k)^2
 \end{equation}
 for any $y=(y_1,\cdots,y_n)\in \mathbb{R}^n$, and
\begin{equation}\label{s2:f-invcon}
\frac{\dot{f}^k-\dot{f}^l}{\kappa_k-\kappa_l}+\frac{\dot{f}^k}{\kappa_l}+\frac{\dot{f}^l}{\kappa_k}\geq~0,\quad \forall~k\neq l.
\end{equation}
\item[(ii)] If $f=f(\kappa_1,\cdots,\kappa_n)$ is inverse concave, then
  \begin{equation}\label{s2:f-invcon-0}
    \sum_{i=1}^n\dot{f}^i\kappa_i^2~\geq~f^2.
  \end{equation}
\end{itemize}
 \end{lem}

\subsection{Evolution equations}
Along any smooth flow
\begin{equation}\label{s2:flow}
  \frac{\partial}{\partial t}X(x,t)=~\varphi(x,t)\nu(x,t)
\end{equation}
of hypersurfaces in hyperbolic space $\mathbb{H}^{n+1}$, where $\varphi$ is a smooth function on the evolving hypersurfaces $M_t=X(M^n,t)$, we have the following evolution equations on the induced metric $g_{ij}$, the induced area element $d\mu_t$ and the Weingarten matrix $\mathcal{W}=(h_i^j)$ of $M_t$:
\begin{align}
  \frac{\partial}{\partial t}g_{ij} =&~ 2\varphi h_{ij} \label{evl-g}\\
    \frac{\partial}{\partial t}d\mu_t=& ~nE_1 \varphi d\mu_t\label{evl-dmu}\\
  \frac{\partial}{\partial t}h_i^j=&~-\nabla^j\nabla_i\varphi-\varphi (h_i^kh_k^j-\delta_i^j)\label{evl-h}
\end{align}
From the evolution equations \eqref{evl-dmu} and \eqref{evl-h}, we can derive the evolution equation of the curvature integral $V_{n-k}$:
\begin{align}\label{s2:evl-Vk}
   \frac d{dt}V_{n-k}(\Omega_t)=&\frac d{dt}\int_{M_t}E_kd\mu_t\nonumber\\
   =&\int_{M_t}\left(\frac{\partial}{\partial t} E_k+nE_1E_k\varphi\right)d\mu_t\nonumber\\
   =& \int_{M_t}\left(-\frac{\partial E_k}{\partial h_i^j}\nabla^j\nabla_i\varphi -\varphi \frac{\partial E_k}{\partial h_i^j}(h_i^kh_k^j-\delta_i^j)+nE_1E_k\varphi\right)d\mu_t\displaybreak[0]\nonumber\\
  = & \int_{M_t}\varphi \biggl((n-k)E_{k+1}+kE_{k-1}\biggr)d\mu_t,
\end{align}
where we used integration by part and the fact that $\frac{\partial E_k}{\partial h_i^j}$ is divergence free. Since the quermassintegrals are related to the curvature integrals by the equations \eqref{s1:quermass-1} and \eqref{s1:quermass-2}, applying an induction argument to the equation \eqref{s2:evl-Vk} yields that
\begin{lem}[cf.\cite{And-Wei2017-2,WX}]\label{s2:lem2}
Along the flow \eqref{s2:flow}, the quermassintegral $W_k$ of the evolving domain $\Omega_t$ satisfies
\begin{equation*}
  \frac d{dt}W_k(\Omega_t)~=~\int_{M_t}E_k(\kappa)\varphi d\mu_t,\quad k=0,\cdots,n.
\end{equation*}
\end{lem}
We can also derive the following evolution equation for the modified quermassintegrals.
\begin{lem}\label{s2:lem3}
Along the flow \eqref{s2:flow}, the modified quermassintegral $\widetilde{W}_k$ of the evolving domain $\Omega_t$ satisfies
\begin{equation*}
  \frac d{dt}\widetilde{W}_k(\Omega_t)~=~\int_{M_t}E_k(\lambda)\varphi d\mu_t,\quad k=0,\cdots,n,
\end{equation*}
 where $\lambda=(\lambda_1,\cdots,\lambda_n)=(\kappa_1-1,\cdots,\kappa_n-1)$ are the shifted principal curvatures of $M_t$.
\end{lem}
\proof
Firstly, we derive the formula for $\sigma_{k}(\lambda)$ in terms of $\sigma_i(\kappa), i=0,\cdots,k$. By the definition of the elementary symmetric polynomial, we have
\begin{equation*}
  \prod_{i=1}^n(t+\lambda_i)~=\sum_{k=0}^n\sigma_k(\lambda)t^{n-k}.
\end{equation*}
On the other hand,
\begin{align*}
   \prod_{i=1}^n(t+\lambda_i) =&~\prod_{i=1}^n(t-1+\kappa_i) ~=~ \sum_{l=0}^n\sigma_l(\kappa)(t-1)^{n-l} \nonumber\\
  = &~ \sum_{l=0}^n\sigma_l(\kappa)\sum_{i=0}^{n-l}\binom{n-l}it^i(-1)^{n-l-i} \nonumber\\
  =&~\sum_{k=0}^n\left(\sum_{i=0}^k\binom{n-i}{k-i}(-1)^{k-i}\sigma_i(\kappa)\right)t^{n-k}.
\end{align*}
Comparing the coefficients of $t^{n-k}$, we have
\begin{align*}
  \sigma_k(\lambda)~=&~\sum_{i=0}^k\binom{n-i}{k-i}(-1)^{k-i}\sigma_i(\kappa)\nonumber\\
  =&~\sum_{i=0}^k\binom{n-i}{k-i}(-1)^{k-i}\binom{n}{i}E_i(\kappa)\nonumber\\
  =&~\binom{n}k\sum_{i=0}^k(-1)^{k-i}\binom{k}{i}E_i(\kappa).
\end{align*}
Equivalently, we have
\begin{equation}\label{s2:sk-3}
  E_k(\lambda)~=~{\binom{n}{k}}^{-1}\sigma_k(\lambda)~=~\sum_{i=0}^k(-1)^{k-i}\binom{k}{i}E_i(\kappa).
\end{equation}
Then by the definition \eqref{s1:Wk-td} of $\widetilde{W}_k$ and Lemma \ref{s2:lem2},
\begin{align*}
  \frac d{dt}\widetilde{W}_k(\Omega_t)~=&\sum_{i=0}^k(-1)^{k-i}\binom{k}{i}\frac d{dt}W_i(\Omega_t)\nonumber\displaybreak[0]\\
  =&\sum_{i=0}^k(-1)^{k-i}\binom{k}{i}\int_{M_t}E_i(\kappa) \varphi d\mu_t ~=~ \int_{M_t}E_k(\lambda)\varphi d\mu_t.
\end{align*}
\endproof

If we consider the flow \eqref{flow-VMCF}, i.e., $\varphi=\phi(t)-\Psi(\mathcal{W})$, using \eqref{evl-h} and the Simons' identity we have the evolution equations for the curvature function $\Psi=\Psi(\mathcal{W})$ and the Weingarten matrix $\mathcal{W}=(h_i^j)$ of $M_t$ (see \cite{And-Wei2017-2}):
\begin{equation}\label{s2:evl-Psi}
  \frac{\partial}{\partial t}\Psi=~\dot{\Psi}^{kl}\nabla_k\nabla_l\Psi+(\Psi-\phi(t))(\dot{\Psi}^{ij}h_i^kh_k^j-\dot{\Psi}^{ij}\delta_i^j),
\end{equation}
and
\begin{align}\label{s2:evl-h}
 \frac{\partial}{\partial t}h_i^j  =& \dot{\Psi}^{kl}\nabla_k\nabla_lh_i^j+\ddot{\Psi}^{kl,pq}\nabla_ih_{kl}\nabla^jh_{pq} +(\dot{\Psi}^{kl}h_k^rh_{rl}+\dot{\Psi}^{kl}g_{kl})h_i^j\nonumber\\
   &\quad -\dot{\Psi}^{kl}h_{kl}(h_i^ph_{p}^j+\delta_i^j)+(\Psi-\phi(t))(h_i^kh_{k}^j-\delta_i^j),
\end{align}
where $\nabla$ denotes the Levi-Civita connection with respect to the induced metric $g_{ij}$ on $M_t$, and $\dot{\Psi}^{kl}, \ddot{\Psi}^{kl,pq}$ denote the derivatives of $\Psi$ with respect to the components of the Weingarten matrix $\mathcal{W}=(h_i^j)$.

If we consider the flow \eqref{flow-VMCF-2} of h-convex hypersurfaces, i.e., $\varphi=\phi(t)-F(\mathcal{W}-\mathrm{I})$,  we have the similar evolution equation for the curvature function $F$
\begin{align}\label{s2:evl-F}
  \frac{\partial}{\partial t}F=&~\dot{F}^{kl}\nabla_k\nabla_lF+(F-\phi(t))(\dot{F}^{ij}h_i^kh_k^j-\dot{F}^{ij}\delta_i^j),
\end{align}
and a parabolic type equation for the Weingarten matrix $\mathcal{W}=(h_i^j)$ of $M_t$:
\begin{align}\label{s2:evl-h2}
 \frac{\partial}{\partial t}h_i^j  =& \dot{F}^{kl}\nabla_k\nabla_lh_i^j+\ddot{F}^{kl,pq}\nabla_ih_{kl}\nabla^jh_{pq} +(\dot{F}^{kl}h_k^rh_{rl}+\dot{F}^{kl}g_{kl})h_i^j\nonumber\\
   &\quad -\dot{F}^{kl}h_{kl}(h_i^ph_{p}^j+\delta_i^j)+(F-\phi(t))(h_i^kh_{k}^j-\delta_i^j).
\end{align}
However, we observe that $\dot{F}^{kl}, \ddot{F}^{kl,pq}$ in \eqref{s2:evl-F} and \eqref{s2:evl-h2} denote the derivatives of $F$ with respect to the components of shifted Weingarten matrix $\mathcal{W}-\mathrm{I}$, so the homogeneity of $F$ implies that $\dot{F}^{kl}(h_{k}^l-\delta_k^l)=F$. Denote the components of the shifted Weingarten matrix by $S_{ij}=h_i^j-\delta_i^j$. Then the equation \eqref{s2:evl-h2} implies that
\begin{align}\label{s2:evl-S}
 \frac{\partial}{\partial t}S_{ij}  =&~ \dot{F}^{kl}\nabla_k\nabla_lS_{ij}+\ddot{F}^{kl,pq}\nabla_ih_{kl}\nabla^jh_{pq} +(\dot{F}^{kl}S_{kr}S_{rl}+2F-2\phi(t))S_{ij}\nonumber\\
   &\quad -(\phi(t)+\dot{F}^{kl}\delta_k^l) S_{ik}S_{kj}+\dot{F}^{kl}S_{kr}S_{rl}\delta_i^j.
\end{align}

\subsection{A generalised maximum principle}

In \S \ref{sec:Pinch}, we will use the tensor maximum principle to prove that the pinching estimate along the flow \eqref{flow-VMCF-2}. For the convenience of readers, we include here the statement of the tensor maximum principle, which was first proved by Hamilton \cite{Ha1982} and was generalized by the first author in \cite{And2007}.
\begin{thm}[\cite{And2007}]\label{s2:tensor-mp}
Let $S_{ij}$ be a smooth time-varying symmetric tensor field on a compact manifold $M$, satisfying
\begin{equation*}
\frac{\partial}{\partial t}S_{ij}=a^{kl}\nabla_k\nabla_lS_{ij}+u^k\nabla_kS_{ij}+N_{ij},
\end{equation*}
where $a^{kl}$ and $u$ are smooth, $\nabla$ is a (possibly time-dependent) smooth symmetric connection, and $a^{kl}$ is positive definite everywhere. Suppose that
\begin{equation}\label{s2:TM2}
  N_{ij}v^iv^j+\sup_{\Lambda}2a^{kl}\left(2\Lambda_k^p\nabla_lS_{ip}v^i-\Lambda_k^p\Lambda_l^qS_{pq}\right)\geq 0
\end{equation}
whenever $S_{ij}\geq 0$ and $S_{ij}v^j=0$. If $S_{ij}$ is positive definite everywhere on $M$ at $t=0$ and on $\partial M$ for $0\leq t\leq T$, then it is positive on $M\times[0,T]$.
\end{thm}

\section{Preserving positive sectional curvature}\label{sec:PSC}
In this section, we will prove that the flow \eqref{flow-VMCF} preserves the positivity of sectional curvatures, if $\alpha>0$ and $f$ satisfies Assumption \ref{s1:Asum}.
\begin{thm}\label{s3:thm-2}
If the initial hypersurface $M_0$ has positive sectional curvature, then along the flow \eqref{flow-VMCF} in $\mathbb{H}^{n+1}$ with $f$ satisfying Assumption \ref{s1:Asum} and any power $\alpha>0$  the evolving hypersurface $M_t$ has positive sectional curvature for $t>0$.
\end{thm}
\proof
The sectional curvature defines a smooth function on the Grassmannian bundle of two-dimensional subspaces of $TM$.  For convenience we lift this to a function on the orthonormal frame bundle $O(M)$ over $M$:  Given
a point $x\in M$ and $t\geq 0$, and a frame ${\mathbb O} = \{e_1,\cdots,e_n\}$ for $T_xM$ which is orthonormal with respect to $g(x,t)$, we define
$$
G(x,t,{\mathbb O}) = h_{(x,t)}(e_1,e_1)h_{(x,t)}(e_2,e_2)-h_{(x,t)}(e_1,e_2)^2-1.
$$
We consider a point $(x_0,t_0)$ and a frame ${\mathbb O}_0 = \{\bar e_1,\cdots,\bar e_n\}$ at which a new minimum of the function $G$ is attained, so that we have $G(x,t,{\mathbb O})\geq G(x_0,t_0,{\mathbb O}_0)$ for all $x\in M$ and all $t\in[0,t_0]$, and all ${\mathbb O}\in F(M)_{(x,t)}$.  The fact that ${\mathbb O}_0$ achieves the minimum of $G$ over the fibre $F(M)_{(x_0,t_0)}$ implies that $e_1$ and $e_2$ are eigenvectors of $h_{(x_0.t_0)}$ corresponding to $\kappa_1$ and $\kappa_2$, where $\kappa_1\leq\kappa_2\leq\cdots\leq\kappa_n$ are the principal curvatures at $(x_0,t_0)$.  Since $G$ is invariant under rotation in the subspace orthogonal to $\bar e_1$ and $\bar e_2$, we can assume that $h(\bar e_i,\bar e_i)=\kappa_i$ and $h(\bar e_i,\bar e_j)=0$ for $i\neq j$.

The time derivative of $G$ at $(x_0,t_0,{\mathbb O}_0)$ is given by Equation \eqref{s2:evl-h}, noting that the frame ${\mathbb O}(t)$ for $T_xM$ defined by $\frac{d}{dt}e_i(t) = (F^\alpha-\phi){\mathcal W}(e_i)$ remains orthonormal with respect to $g(x,t)$ if $e_i(t_0)=\bar e_i$ for each $i$.  This yields the following:
\begin{align}\label{eq:dGdt}
\frac{\partial}{\partial t}G|_{(x_0,t_0,{\mathbb O}_0)}
&= \kappa_1\frac{\partial}{\partial t}h_{2}^{2}+\kappa_2\frac{\partial}{\partial_t}h_{1}^{1}\notag\\
&=\kappa_1\dot\Psi^{kl}\nabla_k\nabla_lh_{22}+\kappa_2\dot\Psi^{kl}\nabla_k\nabla_lh_{11}+
\kappa_1\ddot\Psi(\nabla_2h,\nabla_2h)+\kappa_2\ddot\Psi(\nabla_1h,\nabla_1h)\notag\\
&\quad\null+2\left(\dot{\Psi}^{kl}h_k^rh_{rl}+\dot{\Psi}^{kl}g_{kl}\right)\kappa_1\kappa_2-(\alpha-1)\Psi \kappa_1\kappa_2(\kappa_1+\kappa_2)\notag\\
&\quad\null-(\alpha+1)\Psi(\kappa_1+\kappa_2)-\phi(t)(\kappa_1\kappa_2-1)(\kappa_1+\kappa_2).
\end{align}

%
Since $\Psi=f^{\alpha}$, we have the following:
\begin{align*}
2&\left(\dot{\Psi}^{kl}h_k^rh_{rl}+\dot{\Psi}^{kl}g_{kl}\right)\kappa_1\kappa_2-(\alpha-1)\Psi \kappa_1\kappa_2(\kappa_1+\kappa_2)-(\alpha+1)\Psi(\kappa_1+\kappa_2)-\phi(t)(\kappa_1\kappa_2-1)(\kappa_1+\kappa_2)\\
&= 2\alpha f^{\alpha-1}\sum_k\dot f^k(\kappa_k-\kappa_2)(\kappa_k-\kappa_1)+G\left(f^{\alpha-1}\sum_k\dot f^k\kappa_k(2\alpha\kappa_k-(\alpha-1)(\kappa_1+\kappa_2))-\phi(t)(\kappa_1+\kappa_2)
\right)\\
&\geq -CG,
\end{align*}
where $C$ is a bound for the smooth function in the last bracket.
To estimate the remaining terms, we consider the second derivatives of $G$ along a curve on $O(M)$ defined as follows:  We let $\gamma$ be any geodesic of $g(t_0)$ in $M$ with $\gamma(0)=x_0$, and define a frame ${\mathbb O}(s)=(e_1(s),\cdots,e_n(s))$ at $\gamma(s)$ by taking $e_i(0)=\bar e_i$ for each $i$, and $\nabla_se_i(s) = \Gamma_{ij}e_j(s)$ for some constant antisymmetric matrix $\Gamma$.  Then we compute
\begin{align}\label{eq:d2Gds2}
\frac{d^2}{ds^2}G(x(s),t_0,{\mathbb O}(s))\Big|_{s=0}
&=\kappa_2\nabla^2_sh_{11}+\kappa_1\nabla^2_sh_{22}+2\left(\nabla_sh_{22}\nabla_sh_{11}-(\nabla_sh_{12})^2\right)\notag\\
&\quad\null +4\sum_{p>2}\Gamma_{1p}\kappa_2\nabla_sh_{1p}+4\sum_{p>2}\Gamma_{2p}\kappa_1\nabla_sh_{2p}\notag\\
&\quad\null +2\sum_{p>2}\Gamma_{1p}^2\kappa_2(\kappa_p-\kappa_1)+2\sum_{p>2}\Gamma_{2p}^2\kappa_1(\kappa_p-\kappa_2).
\end{align}
Since $G$ has a minimum at $(x_0,t_0,{\mathbb O}_0)$, the right-hand side of \eqref{eq:d2Gds2} is non-negative for any choice of $\Gamma$.  Minimizing over $\Gamma$ gives
\begin{align}\label{eq:D2G}
0&\leq \kappa_2\nabla^2_sh_{11}+\kappa_1\nabla^2_sh_{22}+2\left(\nabla_sh_{22}\nabla_sh_{11}-(\nabla_sh_{12})^2\right)\notag\\
&\quad\null -2\sum_{p>2}\frac{\kappa_2}{\kappa_p-\kappa_1}(\nabla_sh_{1p})^2
-2\sum_{p>2}\frac{\kappa_1}{\kappa_p-\kappa_2}(\nabla_sh_{2p})^2
\end{align}
where we terms on the last line as vanishing if the denominators vanish (since the corresponding component of $\nabla h$ vanishes in that case).  This gives
\begin{align}\label{eq:dtGineq}
\frac{\partial}{\partial t}G|_{(x_0,t_0,{\mathbb O}_0)}
&\geq \kappa_1\ddot\Psi(\nabla_2h,\nabla_2h)+\kappa_2\ddot\Psi(\nabla_1h,\nabla_1h)-2\sum_k\dot\Psi^k\left(\nabla_kh_{22}\nabla_kh_{11}-(\nabla_kh_{12})^2\right)\notag\\
&\quad\null +2\sum_k\dot\Psi^k\left(\sum_{p>2}\frac{\kappa_2}{\kappa_p-\kappa_1}(\nabla_kh_{1p})^2
+\sum_{p>2}\frac{\kappa_1}{\kappa_p-\kappa_2}(\nabla_kh_{2p})^2\right)-CG
\end{align}
The right-hand side can be expanded using $\Psi=f^\alpha$ and the identity \eqref{s2:F-ddt}:
\begin{align*}
\frac{f^{1-\alpha}}{\alpha}\left(\frac{d}{dt}G+CG\right)
&\geq
\kappa_2\left( \sum_{k,l}\ddot{f}^{kl}\nabla_1h_{kk}\nabla_1h_{ll}+(\alpha-1)\frac{(\nabla_1f)^2}{f}+\sum_{k\neq l}\frac{\dot{f}^k-\dot{f}^l}{\kappa_k-\kappa_l}(\nabla_1h_{kl})^2
\right)\\
&\quad\null +\kappa_1\left(\sum_{k,l} \ddot{f}^{kl}\nabla_2h_{kk}\nabla_2h_{ll}+(\alpha-1)\frac{(\nabla_2f)^2}{f}+\sum_{k\neq l}\frac{\dot{f}^k-\dot{f}^l}{\kappa_k-\kappa_l}(\nabla_2h_{kl})^2
\right)\\
&\quad\null -2\sum_k\dot f^k\left(\nabla_kh_{22}\nabla_kh_{11}-(\nabla_kh_{12})^2\right)\\
&\quad\null+2\sum_k\dot f^k\left(\sum_{p>2}\frac{\kappa_2}{\kappa_p-\kappa_1}(\nabla_kh_{1p})^2
+\sum_{p>2}\frac{\kappa_1}{\kappa_p-\kappa_2}(\nabla_kh_{2p})^2\right).
\end{align*}
Note that by assumption the function $f$ satisfies the inequalities \eqref{s1:asum-1} and \eqref{s1:asum-2}. By the inequality \eqref{s1:asum-1}, for any $k\neq l$ we have
\begin{align*}
  \frac{\dot{f}^k-\dot{f}^l}{\kappa_k-\kappa_l}+\frac{\dot{f}^k}{\kappa_l}=&~\frac{\dot{f}^k\kappa_k-\dot{f}^l\kappa_l}{(\kappa_k-\kappa_l)\kappa_l}~\geq 0
\end{align*}
The inequality \eqref{s1:asum-2} is equivalent to
\begin{equation*}
  \sum_{k,l}\ddot{f}^{kl}y_ky_l\geq~f^{-1}(\sum_{k=1}^n\dot{f}^ky_k)^2-\sum_{k=1}^n\frac{\dot{f}^k}{\kappa_k}y_k^2
\end{equation*}
for all $(y_1,\cdots,y_n)\in \mathbb{R}^n$.
These imply that
\begin{align}\label{s3:Q1}
\frac{f^{1-\alpha}}{\alpha}\left(\frac{dG}{dt}+CG\right) &\geq
\kappa_2\left( \alpha \frac{(\nabla_1f)^2}{f}-\sum_{k=1}^n\frac{\dot{f}^k}{\kappa_k}(\nabla_1h_{kk})^2-\sum_{k\neq l}\frac{\dot{f}^k}{\kappa_l}(\nabla_1h_{kl})^2\right)\nonumber\\
&\quad\null+\kappa_1\left( \alpha \frac{(\nabla_2f)^2}{f}-\sum_{k=1}^n\frac{\dot{f}^k}{\kappa_k}(\nabla_2h_{kk})^2-\sum_{k\neq l}\frac{\dot{f}^k}{\kappa_l}(\nabla_2h_{kl})^2\right)\nonumber\displaybreak[0]\\
&\quad+2\sum_k\dot{f}^k\left(-\nabla_kh_{11}\nabla_kh_{22}+(\nabla_kh_{12})^2\right)\nonumber\\
&\quad+2\sum_{k=1}^n\sum_{p>2}\frac{\dot{f}^k}{\kappa_p}\left(\kappa_2(\nabla_1h_{kp})^2+\kappa_1(\nabla_2h_{kp})^2\right)\nonumber\\
=&~\kappa_2\alpha\frac{(\nabla_1f)^2}{f}+\kappa_1\alpha\frac{(\nabla_2f)^2}{f}+\sum_{k,p=3}^n\frac{\dot{f}^k}{\kappa_p}\left(\kappa_2(\nabla_1h_{kp})^2+\kappa_1(\nabla_2h_{kp})^2\right)\nonumber\\
&-\kappa_2\left(\frac{\dot{f}^1}{\kappa_1}(\nabla_1h_{11})^2+\frac{\dot{f}^2}{\kappa_2}(\nabla_1h_{22})^2+\frac{\dot{f}^1}{\kappa_2}(\nabla_1h_{12})^2+\frac{\dot{f}^2}{\kappa_1}(\nabla_1h_{21})^2\right)\displaybreak[0]\nonumber\\
&-\kappa_1\left(\frac{\dot{f}^1}{\kappa_1}(\nabla_2h_{11})^2+\frac{\dot{f}^2}{\kappa_2}(\nabla_2h_{22})^2+\frac{\dot{f}^1}{\kappa_2}(\nabla_2h_{12})^2+\frac{\dot{f}^2}{\kappa_1}(\nabla_2h_{21})^2\right)\nonumber\\
&+2\dot{f}^1\left(-\nabla_1h_{11}\nabla_1h_{22}+(\nabla_1h_{12})^2\right)+2\dot{f}^2\left(-\nabla_2h_{11}\nabla_2h_{22}+(\nabla_2h_{12})^2\right)\nonumber\\
&+2\dot f^1\sum_{p>2}\left(\frac{\kappa_2}{\kappa_p}(\nabla_1h_{1p})^2+\frac{\kappa_1}{\kappa_p}(\nabla_2h_{1p})^2\right)\nonumber\\
&+2\dot f^2\sum_{p>2}\left(\frac{\kappa_2}{\kappa_p}(\nabla_1h_{2p})^2+\frac{\kappa_1}{\kappa_p}(\nabla_2h_{2p})^2\right)\displaybreak[0]\nonumber\\
&+2\sum_{k>2}\dot{f}^k\left(-\nabla_kh_{11}\nabla_kh_{22}+(\nabla_kh_{12})^2\right).
\end{align}
Since $(x_0,{\mathbb O}_0)$ is a minimum point of $G$ at time $t_0$, we have $\nabla_iG=0$ for $i=1,\cdots,n$, so
\begin{equation}\label{s3:nab-G}
  \kappa_2\nabla_ih_{11}+\kappa_1\nabla_ih_{22}=0,\qquad i=1,\cdots,n.
\end{equation}
Substituting \eqref{s3:nab-G} into \eqref{s3:Q1}, the second to the fourth lines on the right of \eqref{s3:Q1} vanish,
 the last line becomes $2\sum_{k>2}\dot{f}^k\left(\frac{\kappa_1}{\kappa_2}(\nabla_kh_{22})^2+(\nabla_kh_{12})^2\right)\geq 0$, and the remaining terms are non-negative.

We conclude that $\frac{\partial}{\partial t}G\geq -CG$ at a spatial minimum point, and hence by the maximum principle \cite{Ha1986}*{Lemma 3.5} we have $G\geq \E^{-Ct}\inf_{t=0}G>0$ under the flow \eqref{flow-VMCF}.
\endproof

\section{Proof of Theorem \ref{thm1-1}}\label{sec:thm1-pf}

In this section, we will give the proof of Theorem \ref{thm1-1}.
\subsection{Shape estimate}\label{sec:4-1}
First, we show that the preservation of the volume of $\Omega_t$, together with a reflection argument, implies that the inner radius and outer radius of $\Omega_t$ are uniformly bounded from above and below by positive constants.
\begin{lem}
Denote $\rho_-(t), \rho_+(t)$ be the inner radius and outer radius of $\Omega_t$, the domain enclosed by $M_t$. Then there exist positive constants $c_1,c_2$ depending only on $n,M_0$ such that
\begin{equation}\label{s4:io-radius1}
  0<c_1\leq \rho_-(t)\leq \rho_+(t)\leq c_2
\end{equation}
for all time $t\in [0,T)$.
\end{lem}
\proof
We first use the Alexandrov reflection method to estimate the diameter of $\Omega_t$. In \cite{And-Wei2017-2}, the first and the third authors have already used the Alexandrov reflection method in the proof of convergence of the flow. Let $\gamma$ be a geodesic line in $\mathbb{H}^{n+1}$, and let $H_{\gamma(s)}$ be the totally geodesic hyperbolic $n$-plane in $\mathbb{H}^{n+1}$ which is perpendicular to $\gamma$ at $\gamma(s), s\in \mathbb{R}$. We use the notation $H_{s}^+$ and $H_s^-$ for the half-spaces in $\mathbb{H}^{n+1}$ determined by $H_{\gamma(s)}$:
\begin{equation*}
  H_s^+:=\bigcup_{s'\geq s}H_{\gamma(s')},\qquad  H_s^-:=\bigcup_{s'\leq s}H_{\gamma(s')}
\end{equation*}
For a bounded domain $\Omega$ in $\mathbb{H}^{n+1}$, denote $\Omega^+(s)=\Omega\cap H_s^+$ and $\Omega^-(s)=\Omega\cap H_s^-$. The reflection map across $H_{\gamma(s)}$ is denoted by $R_{\gamma,s}$. We define
\begin{equation*}
  S_{\gamma}^+(\Omega):=\inf\{s\in \mathbb{R}~|~R_{\gamma,s}(\Omega^+(s))\subset\Omega^-(s)\}.
\end{equation*}
It has been proved in \cite{And-Wei2017-2} that for any geodesic line $\gamma$ in $\mathbb{H}^{n+1}$, $S_{\gamma}^+(\Omega_t)$ is strictly decreasing along the flow \eqref{flow-VMCF} unless $R_{\gamma,\bar s}(\Omega_t)=\Omega_t$ for some $\bar s\in \mathbb{R}$. Note that to prove this property, we only need the convexity of the evolving hypersurface $M_t=\partial\Omega_t$ which is guaranteed by the positivity of the sectional curvature. The readers may refer to \cite{Chow97,Chow-Gul96} for more details on the Alexandrov reflection method. 
\begin{figure}
  \centering
  \includegraphics[width=\textwidth]{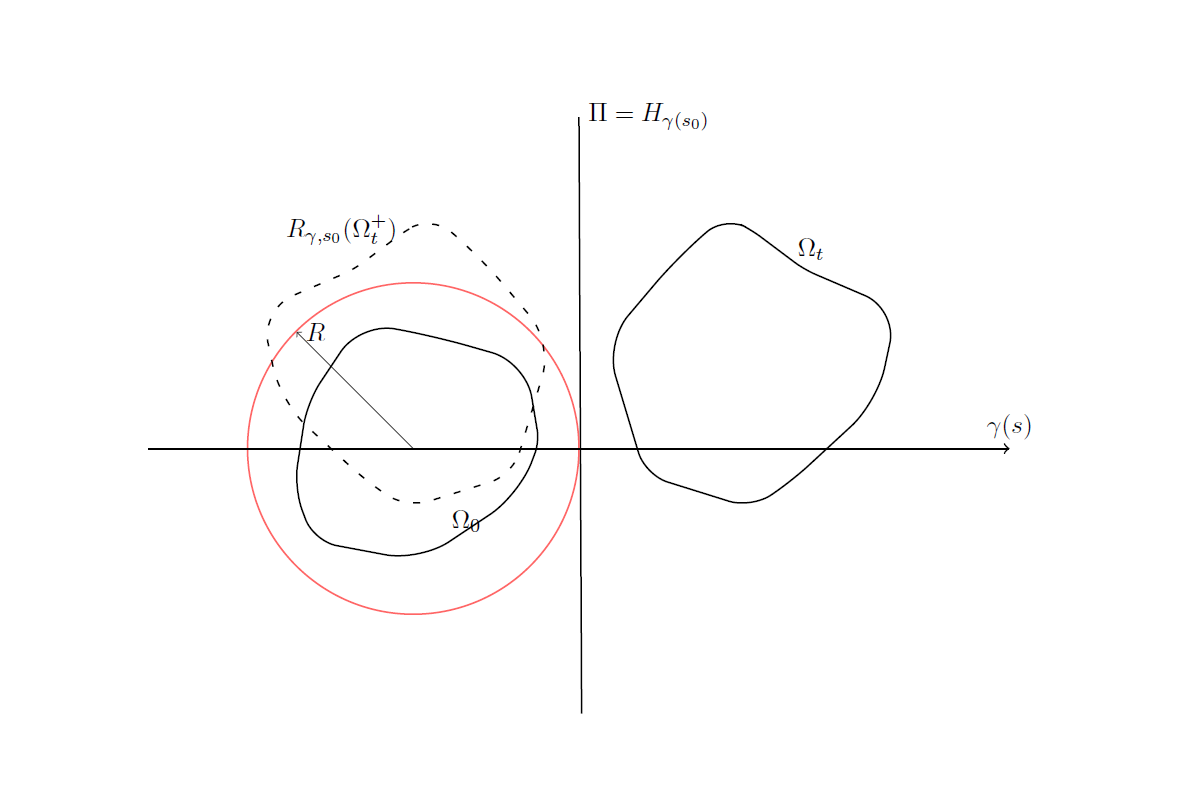}\\
  \caption{$\Omega_t$ can not leave out $B_R$}
\end{figure}

Choose $R>0$ such that the initial domain $\Omega_0$ is contained in some geodesic ball $B_R(p)$ of radius $R$ and centered at some point $p$ in the hyperbolic space. The above reflection property implies that $\Omega_t\cap B_R(p)\neq \emptyset$ for any $t\in [0,T)$. If not, there exists some time $t$ such that $\Omega_t$ does not intersect the geodesic ball $B_R$. Choose a geodesic line $\gamma(s)$ with the property that there exists a geodesic hyperplace $\Pi=H_{\gamma(s_0)}$ which is perpendicular to $\gamma(s)$ and is tangent to the geodesic sphere $\partial B_R$, and the domain $\Omega_t$ lies in the half-space $H^+_{s_0}$. Then $R_{\gamma,s_0}(\Omega^+_0)=\emptyset\subset \Omega_0^-$. Since $S_{\gamma}^+(\Omega_t)$ is decreasing, we have $R_{\gamma,s_0}(\Omega_t^+)\subset \Omega_t^-$. However, this is not possible because $\Omega_t^-=\Omega_t\cap H_{s_0}^-=\emptyset$ and $R_{\gamma,s_0}(\Omega_t^+)$ is obviously not empty.

\begin{figure}
  \centering
  \includegraphics[width=\textwidth]{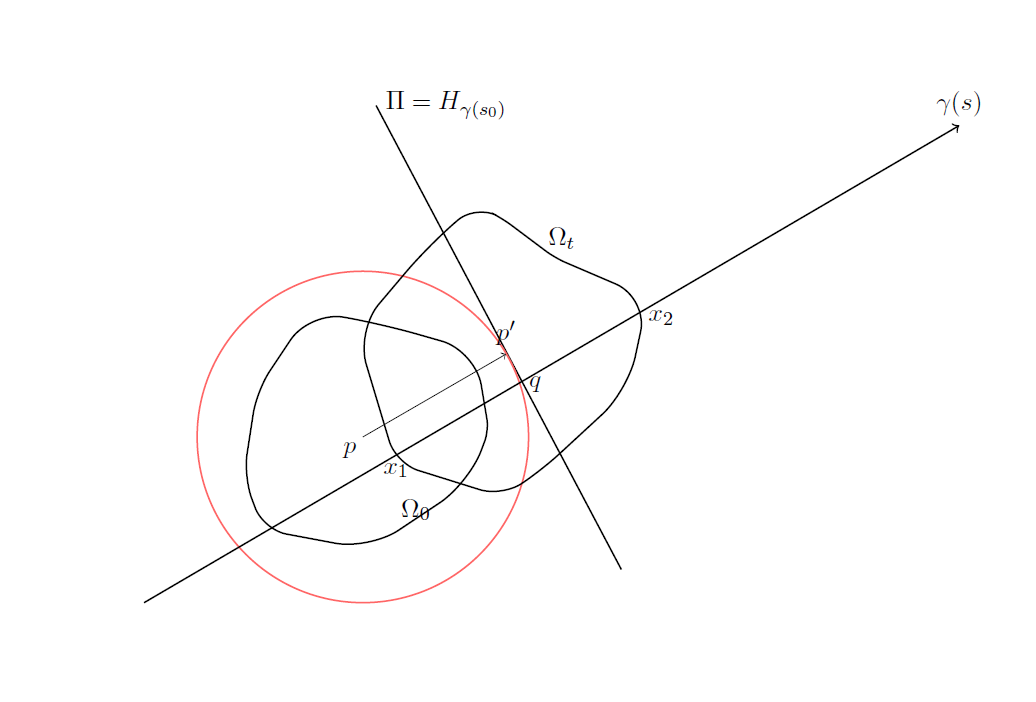}\\
  \caption{Diameter of $\Omega_t$ is bounded}
\end{figure}
For any $t\in [0,T)$, let $x_1, x_2$ be points on $M_t=\partial\Omega_t$ such that $d(p,x_1)=\min\{d(p,x): x\in M_t\}$ and $d(p,x_2)=\max\{d(p,x): x\in M_t\}$, where $d(\cdot,\cdot)$ is the distance in the hyperbolic space. Since $\Omega_0$ is contained in the geodesic ball $B_R(p)$ and $\Omega_t\cap B_R(p)\neq \emptyset$, we deduce from $|\Omega_t|=|\Omega_0|$ that $x_1\in B_R(p)$. If $x_2\in B_R(p)$, then the diameter of $\Omega_t$ is bounded from above by $R$. Therefore it suffices to study the case $x_2\notin B_R(p)$. Let $\gamma(s)$ be the geodesic line passing through $x_1$ and $x_2$, i.e., there are numbers $s_1<s_2\in \mathbb{R}$ such that $\gamma(s_1)=x_1$ and $\gamma(s_2)=x_2$. We choose the geodesic plane $\Pi=H_{\gamma(s_0)}$ for some number $s_0\in (s_1,s_2)$ such that $\Pi$ is perpendicular to $\gamma$ and is tangent to the boundary of $B_R(p)$ at $p'\in\partial B_R(p)$. Let $q=\gamma(s_0)$ be intersection point $\gamma\cap\Pi$. By the Alexandrov reflection property, $d(x_2,q)\leq d(q,x_1)$. Then the triangle inequality implies
\begin{align*}
  d(p,x_2) \leq &~ d(p,x_1)+d(x_1,x_2) \\
  \leq  &~ d(p,x_1)+2d(q,x_1)\\
  \leq &~ d(p,x_1)+2(d(q,p')+d(p',p)+d(p,x_1))\leq ~7R,
\end{align*}
where we used the fact $x_1\in B_R(p)$. This shows that the diameter of $\Omega_t$ is uniformly bounded along the flow \eqref{flow-VMCF}.

To estimate the lower bound of the inner radius of $\Omega_t$, we project the domain $\Omega_t$ in the hyperbolic space $\mathbb{H}^{n+1}$ to the unit ball in Euclidean space $\mathbb{R}^{n+1}$ as in  \cite[\S 5]{And-Wei2017-2}. Denote by $\mathbb{R}^{1,n+1}$ the Minkowski spacetime, that is the vector space $\mathbb{R}^{n+2}$ endowed with the Minkowski spacetime metric $\langle \cdot,\cdot\rangle$ given by $ \langle X,X\rangle=-X_0^2+\sum_{i=1}^nX_i^2$ for any vector  $X=(X_0,X_1,\cdots,X_n)\in \mathbb{R}^{n+2}$. Then the hyperbolic space is characterized as
\begin{equation*}
  \mathbb{H}^{n+1}=~\{X\in \mathbb{R}^{1,n+1},~~\langle X,X\rangle=-1,~X_0>0\}
\end{equation*}
An embedding $X:M^n\to \mathbb{H}^{n+1}$ induces an embedding $Y:M^n\to B_1(0)\subset \mathbb{R}^{n+1}$ by
\begin{equation*}
  X~=~\frac{(1,Y)}{\sqrt{1-|Y|^2}}.
\end{equation*}
The induced metrics $g_{ij}^X$ and $g_{ij}^Y$ of $X(M^n)\subset \mathbb{H}^{n+1}$ and $Y(M^n)\subset \mathbb{R}^{n+1}$ are related by
\begin{align*}
  g_{ij}^X   =&\frac 1 {1-|Y|^2}\left(g_{ij}^Y+\frac{\langle Y,\partial_iY\rangle\langle  Y,\partial_jY\rangle }{(1-|Y|^2)}\right)
\end{align*}
Let $\tilde{\Omega}_t\subset B_1(0)$ be the corresponding image of $\Omega_t$ in $B_1(0)\subset \mathbb{R}^{n+1}$, and observe that $\tilde\Omega_t$ is a convex Euclidean domain.   Then the diameter bound of $\Omega_t$ implies the diameter bound on $\tilde{\Omega}_t$. In particular, $|Y|\leq C<1$ for some constant $C$. This implies that the induced metrics $g_{ij}^X$ and $g_{ij}^Y$ are comparable. So the volume of $\tilde{\Omega}_t$  is also bounded below by a constant depending on the volume of $\Omega_0$ and the diameter of $\Omega_t$.  Let $\omega_{\min}(\tilde{\Omega}_t)$ be the minimal width of $\tilde\Omega_t$.  Then the volume of $\tilde{\Omega}_t$ is bounded by the a constant times the $\omega_{\min}(\tilde\Omega_t)(\text{\rm diam}(\tilde\Omega_t)^{n}$, since $\tilde\Omega_t$ is contained in a spherical prism of height $\omega_{\min}(\tilde\Omega_t)$ and radius $\text{\rm diam}(\tilde\Omega_t)$.  It follows that $\omega_{\min}(\tilde{\Omega}_t)$ is bounded from below by a positive constant $C$. Since $\tilde{\Omega}_t$ is strictly convex, an estimate of Steinhagen \cite{Steinhagen} implies that the inner radius $\tilde{\rho}_-(t)$ of $\tilde{\Omega}_t$ is bounded below by $\tilde{\rho}_-(t)\geq c(n)\omega_{\min}\geq C>0$, from which we obtain the uniform positive lower bound on the inner radius $\rho_-(t)$ of $\Omega_t$. This finishes the proof.
\endproof

By \eqref{s4:io-radius1}, the inner radius of $\Omega_t$ is bounded below by a positive constant $c_1$. This implies that there exists a geodesic ball of radius $c_1$ contained in $\Omega_t$ for each $t\in [0,T)$. The same argument as in \cite[Lemma 4.2]{And-Wei2017-2} yields the existence of a geodesic ball with fixed center enclosed by the flow hypersurface on a suitable fixed time interval.

\begin{lem}\label{s4:lem-inball}
Let $M_t$ be a smooth solution of the flow \eqref{flow-VMCF} on $[0,T)$ with the global term $\phi(t)$ given by \eqref{s1:phit}. For any $t_0\in [0,T)$, let $B_{\rho_0}(p_0)$ be the inball of $\Omega_{t_0}$, where $\rho_0=\rho_-(t_0)$. Then
\begin{equation}\label{s4:inball-eqn1}
  B_{\rho_0/2}(p_0)\subset \Omega_t,\quad t\in [t_0, \min\{T,t_0+\tau\})
\end{equation}
for some $\tau$ depending only on $n,\alpha,\Omega_0$.
\end{lem}

\subsection{Upper bound of $F$}\label{sec:4-2}

Now we can use the technique of Tso \cite{Tso85} as in \cite{And-Wei2017-2} to prove the upper bound of $F$ along the flow \eqref{flow-VMCF} provided that $F$ satisfies  Assumption \ref{s1:Asum}. The inequality  \eqref{s1:asum-1} and the fact that each $M_t$ has positive sectional curvature are crucial in the proof.

\begin{thm}\label{s4:F-bd}
Assume that $F$  satisfies the Assumption \ref{s1:Asum}. Then along the flow \eqref{flow-VMCF} with any $\alpha>0$, we have $F\leq C$ for any $t\in [0,T)$, where $C$ depends on $n,\alpha,M_0$ but not on $T$.
\end{thm}
\proof
For any given $t_0\in [0,T)$, let $B_{\rho_0}(p_0)$ be the inball of $\Omega_{t_0}$, where $\rho_0=\rho_-(t_0)$.
Consider the support function $u(x,t)=\sinh r_{p_0}(x)\langle \partial r_{p_0},\nu\rangle $ of $M_t$ with respect to the point $p_0$, where $r_{p_0}(x)$ is the distance function in $\mathbb{H}^{n+1}$ from the point $p_0$. Since $M_t$ is strictly convex,  by  \eqref{s4:inball-eqn1},
\begin{equation}\label{s4:sup-1}
  u(x,t)~\geq ~\sinh(\frac{\rho_0}2)~=:~2c
\end{equation}
on $M_t$ for any $t\in[t_0,\min\{T,t_0+\tau\})$. On the other hand, the estimate \eqref{s4:io-radius1} implies that $u(x,t)\leq \sinh(2c_2)$ on $M_t$ for all $t\in[t_0,\min\{T,t_0+\tau\})$. Recall that the support function $u(x,t)$ evolves by
\begin{align}\label{s4:evl-u}
   \frac{\partial}{\partial t}u =& \dot{\Psi}^{kl}\nabla_k\nabla_lu+\cosh r_{p_0}(x)\left(\phi(t)-\Psi-\dot{\Psi}^{kl}h_{kl}\right)+\dot{\Psi}^{ij}h_i^kh_{kj}u.
\end{align}
as we computed in \cite{And-Wei2017-2}, where $\Psi=F^{\alpha}(\mathcal{W})$. Define the auxiliary function
\begin{equation*}
  W(x,t)=\frac {\Psi(x,t)}{u(x,t)-c},
\end{equation*}
which is well-defined on $M_t$ for all $t\in [t_0,\min\{T,t_0+\tau\})$. Combining \eqref{s2:evl-Psi} and \eqref{s4:evl-u}, we have
\begin{align*}
   \frac{\partial}{\partial t}W= &\dot{\Psi}^{ij}\left(\nabla_j\nabla_iW+\frac 2{u-c}\nabla_iu\nabla_jW\right)\nonumber \\
  &\quad-\frac{\phi(t)}{u-c}\left( \dot{\Psi}^{ij}(h_i^kh_{k}^j-\delta_i^j)+W\cosh r_{p_0}(x)\right)\nonumber\\
  &\quad +\frac{\Psi}{(u-c)^2}(\Psi+\dot{\Psi}^{kl}h_{kl})\cosh r_{p_0}(x)-\frac{c\Psi}{(u-c)^2}\dot{\Psi}^{ij}h_i^kh_{k}^j-W\dot{\Psi}^{ij}\delta_i^j.
\end{align*}
By homogeneity of $\Psi$ and the inverse-concavity of $F$, we have $\Psi+\dot{\Psi}^{kl}h_{kl}=(1+\alpha)\Psi$ and $\dot{\Psi}^{ij}h_i^kh_{k}^j\geq \alpha F^{\alpha+1}$. Moreover, by the inequality \eqref{s1:asum-1} and the fact that $\kappa_1\kappa_2>1$, we have
\begin{align*}
  \dot{\Psi}^{ij}(h_i^kh_{k}^j-\delta_i^j)=&~\alpha f^{\alpha-1}\sum_{i=1}^n\dot{f}^i(\kappa_i^2-1)\\
  \geq &~ \alpha f^{\alpha-1}\left(\dot{f}^2(\kappa_2^2-1)+\dot{f}^1(\kappa_1^2-1)\right)\displaybreak[0]\\
  \geq &~\alpha f^{\alpha-1}\dot{f}^1\left(\frac{\kappa_1}{\kappa_2}(\kappa_2^2-1)+(\kappa_1^2-1)\right)\\
  =&~\alpha f^{\alpha-1}\dot{f}^1\kappa_2^{-1}(\kappa_1\kappa_2-1)(\kappa_1+\kappa_2)~\geq~0,
\end{align*}
where we used $\kappa_i\geq 1$ for $i=2,\cdots,n$ in the first inequality.  Then we arrive at
\begin{align}\label{s4:evl-W-1-2}
   \frac{\partial}{\partial t}W\leq & ~\dot{\Psi}^{ij}\left(\nabla_j\nabla_iW+\frac 2{u-c}\nabla_iu\nabla_jW\right)+(\alpha+1)W^2\cosh r_{p_0}(x)-\alpha cW^2F.
\end{align}

The remaining proof of Theorem \ref{s4:F-bd} is the same with \cite[\S 4]{And-Wei2017-2}. We include it here for convenience of the readers. Using \eqref{s4:sup-1} and the upper bound $r_{p_0}(x)\leq 2c_2$, we obtain from \eqref{s4:evl-W-1-2} that
\begin{align}\label{s4:evl-W-3}
   \frac{\partial}{\partial t}W\leq &~ \dot{\Psi}^{ij}\left(\nabla_j\nabla_iW+\frac 2{u-c}\nabla_iu\nabla_jW\right)\nonumber \\
  &\quad +W^2\left((\alpha+1)\cosh (2c_2)-\alpha c^{1+\frac 1{\alpha}}W^{1/{\alpha}}\right)
\end{align}
holds on $[t_0,\min\{T,t_0+\tau\})$. Let $\tilde{W}(t)=\sup_{M_t}W(\cdot,t)$. Then \eqref{s4:evl-W-3} implies that
\begin{equation*}
  \frac d{dt}\tilde{W}(t)\leq \tilde{W}^2\left((\alpha+1)\cosh (2c_2)-\alpha c^{1+\frac 1{\alpha}}\tilde{W}^{1/{\alpha}}\right)
\end{equation*}
from which it follows from the maximum principle that
\begin{equation}\label{s4:evl-W-4}
  \tilde{W}(t)\leq \max\left\{ \left(\frac {2(1+\alpha)\cosh(2c_2)}{\alpha}\right)^{\alpha}c^{-(\alpha+1)}, \left(\frac {2}{1+\alpha}\right)^{\frac{\alpha}{1+\alpha}}c^{-1}(t-t_0)^{-\frac{\alpha}{1+\alpha}}\right\}.
\end{equation}
Then the upper bound on $F$ follows from \eqref{s4:evl-W-4} and the facts that
\begin{equation*}
  c=\frac 12\sinh(\frac{\rho_0}2)\geq \frac 12\sinh(\frac{c_1}2)
\end{equation*}
and $u-c\leq 2c_2$, where $c_1,c_2$ are constants in \eqref{s4:io-radius1} depending only on $n,M_0$.
\endproof

\subsection{Long time existence and convergence}\label{sec:4-3}
In this subsection, we complete the proof of Theorem \ref{thm1-1}. Firstly, in the case (i) of Theorem \ref{thm1-1}, we can deduce directly a uniform estimate on the principal curvatures of $M_t$. In fact, since $f(\kappa)$ is bounded from above by Theorem \ref{s4:F-bd},
\begin{align}\label{s4:f-condi-1}
  C \geq & ~f(\kappa_1,\kappa_2,\cdots,\kappa_n)~ \geq ~f(\kappa_1,\frac 1{\kappa_1},\cdots,\frac 1{\kappa_n}),
\end{align}
where in the second inequality we used the facts that $f$ is increasing in each argument and $\kappa_i\kappa_1>1$ for $i=2,\cdots,n$. Combining \eqref{s4:f-condi-1} and \eqref{thm1-cond1} gives that $\kappa_1\geq C>0$ for some uniform constant $C$.  Since the dual function $f_*$ of $f$ vanishes on the boundary of the positive cone $\Gamma_+$ and $f_*(\tau_i)=1/f(\kappa_i)\geq C>0$ by Theorem \ref{s4:F-bd}, the upper bound on $\tau_i=1/{\kappa_i}\leq C$ gives the lower bound on $\tau_i$, which is equivalent to the upper bound of $\kappa_i$. In summary, we obtain uniform two-sided positive bound on the principal curvatures of $M_t$ along the flow \eqref{flow-VMCF} in the case (i) of Theorem \ref{thm1-1}.

The examples of $f$ satisfying Assumption \ref{s1:Asum} and the condition (i) of Theorem \ref{thm1-1} include
\begin{itemize}
  \item[a).] $n\geq 2$, $f=n^{-1/k}S_k^{1/k}$ with $k>0$;
  \item[b).] $n\geq 3$, $f= E_k^{1/k}$ with  $k=1,\cdots,n$;
  \item[c).] $n=2, f=(\kappa_1+\kappa_2)/2$.
\end{itemize}

We next consider the case (ii) of Theorem \ref{thm1-1}, i.e., $n=2, f=(\kappa_1\kappa_2)^{1/2}$. In general, the estimate $1\leq \kappa_1\kappa_2=f(\kappa)\leq C$ can not prevent $\kappa_2$ from going to infinity. Instead, we will prove that the pinching ratio $\kappa_2/\kappa_1$ is bounded from above along the flow \eqref{flow-VMCF} with $f=(\kappa_1\kappa_2)^{1/2}$ and $\alpha\in[1/2,2]$. This together with the estimate $1\leq \kappa_1\kappa_2\leq C$ yields the uniform estimate on $\kappa_1$ and $\kappa_2$.

\begin{lem}
In the case $n=2, f=(\kappa_1\kappa_2)^{1/2}$ and $\alpha\in [1/2,2]$, the principal curvatures $\kappa_1,\kappa_2$ of $M_t$ satisfy
\begin{equation}\label{s4:lem-pinch}
  0<\frac 1C~\leq \kappa_1\leq \kappa_2\leq C,\quad \forall~t\in [0,T)
\end{equation}
for some positive constant $C$ along the flow \eqref{flow-VMCF}.
\end{lem}
\proof
In this case, $\Psi(\mathcal{W})=\psi(\kappa)=K^{\alpha/2}$, where $K=\kappa_1\kappa_2$ is the Gauss curvature. The derivatives of $\psi$ with respect to $\kappa_i$ are listed in the following:
\begin{align}
  \dot{\psi}^1 =&~ \frac{\alpha}2K^{\frac{\alpha}2-1}\kappa_2,\qquad  \dot{\psi}^2 = \frac{\alpha}2K^{\frac{\alpha}2-1}\kappa_1\label{s4:psi-d}\\
  \ddot{\psi}^{11} =& ~\frac{\alpha}2(\frac{\alpha}2-1)K^{\frac{\alpha}2-2}\kappa_2^2,\qquad
  \ddot{\psi}^{22} =~\frac{\alpha}2(\frac{\alpha}2-1)K^{\frac{\alpha}2-2}\kappa_1^2\\
  \ddot{\psi}^{12} =&~ \ddot{\psi}^{21}=~\frac{\alpha^2}4K^{\frac{\alpha}2-1}.  \label{s4:psi-dd}
\end{align}
Then we have
\begin{align}
  \dot{\Psi}^{ij}h_i^kh_k^j= &~ \sum_{i=1}^n\dot{\psi}^i\kappa_i^2=\frac{\alpha}2K^{\alpha/2}H \label{s4:psi-d1}\\
   \dot{\Psi}^{ij}\delta_i^j= &~ \sum_{i=1}^n\dot{\psi}^i=\frac{\alpha}2K^{\frac{\alpha}2-1}H,\qquad \dot{\Psi}^{ij}h_i^j=\alpha K^{\alpha/2},\label{s4:psi-d2}
\end{align}
where $H=\kappa_1+\kappa_2$ is the mean curvature.

To prove the estimate \eqref{s4:lem-pinch}, we define a function
\begin{equation*}
  G=~K^{\alpha-2}(H^2-4K)
\end{equation*}
and aim to prove that $G(x,t)\leq \max_{M_0}G(x,0)$ by maximum principle. Using \eqref{s4:psi-d1} and \eqref{s4:psi-d2}, the evolution equations \eqref{s2:evl-Psi} and \eqref{s2:evl-h} imply that
\begin{equation}\label{s4:K-evl}
   \frac{\partial}{\partial t}K  =\dot{\Psi}^{kl}\nabla_k\nabla_lK+(\frac{\alpha}2-1)K^{-1}\dot{\Psi}^{kl}\nabla_kK\nabla_lK+(K^{\alpha/2}-\phi(t))(K-1)H.
\end{equation}
and
\begin{align}\label{s4:H-evl}
   \frac{\partial}{\partial t}H =&\dot{\Psi}^{kl}\nabla_k\nabla_lH+\ddot{\Psi}^{kl,pq}\nabla_ih_{kl}\nabla^ih_{pq}\nonumber\\
   &\quad +K^{\frac{\alpha}2-1}(K+\frac{\alpha}2(1-k))(H^2-4K)+2K^{\alpha/2}(K-1)\nonumber\\
   &\quad -\phi(t)(H^2-2K-2).
\end{align}
By a direct computation using \eqref{s4:K-evl} and \eqref{s4:H-evl}, we obtain the evolution equation of $G$ as follows:
\begin{align}\label{s4:G-evl}
   \frac{\partial}{\partial t}G =&\dot{\Psi}^{kl}\nabla_k\nabla_lG-2(\alpha-2)K^{-1}\dot{\Psi}^{kl}\nabla_kK\nabla_lG\nonumber\\
   &\quad +(\alpha-2)(\frac{3\alpha}2-2)K^{\alpha-4}\dot{\Psi}^{kl}\nabla_kK\nabla_lK(H^2-4K)\nonumber\\
   &\quad -2K^{\alpha-2}\dot{\Psi}^{kl}\nabla_kH\nabla_lH-2(\alpha-2)K^{\alpha-3}\dot{\Psi}^{kl}\nabla_kK\nabla_lK\nonumber\\
   &\quad +2HK^{\alpha-2}\ddot{\Psi}^{kl,pq}\nabla_ih_{kl}\nabla^ih_{pq}\nonumber\\
   &\quad +2HK^{\frac{3\alpha}2-3}(H^2-4K)-HK^{\alpha-3}(H^2-4K)(\alpha K+2-\alpha)\phi(t)
\end{align}

We will apply the maximum principle to prove that $\max_{M_t}G$ is non-increasing in time along the flow \eqref{flow-VMCF}. We first look at the zero order terms of \eqref{s4:G-evl}, i.e., the terms in the last line of \eqref{s4:G-evl} which we denote by $Q_0$. Since $K=\kappa_1\kappa_2>1$ by Theorem \ref{s3:thm-2}, we have
\begin{equation*}
  \phi(t)=\frac 1{|M_t|}{\int_{M_t}K^{\alpha/2}}~\geq 1,\quad \mathrm{and}\quad \alpha K+2-\alpha>2.
\end{equation*}
We also have $H^2-4K=(\kappa_1-\kappa_2)^2\geq 0$. Then
\begin{align*}
  Q_0 \leq &~ 2HK^{\frac{3\alpha}2-3}(H^2-4K)-HK^{\alpha-3}(H^2-4K)(\alpha K+2-\alpha) \\
   =& ~HK^{\alpha-3}(H^2-4K)\biggl(2K^{\alpha/2}-\alpha K+\alpha-2\biggr).
\end{align*}
For any $K>1$, denote $f(\alpha)=2K^{\alpha/2}-\alpha K+\alpha-2$. Then $f(2)=f(0)=0$ and $f(\alpha)$ is a convex function of $\alpha$. Therefore $f(\alpha)\leq 0$ and $Q_0\leq 0$ provided that $\alpha\in [0,2]$.

At the critical point of $G$, we have $\nabla_iG=0$ for all $i=1,2$, which is equivalent to
\begin{equation}\label{s4:G-grad}
  2H\nabla_iH=\left(4(\alpha-1)-(\alpha-2)K^{-1}H^2\right)\nabla_iK.
\end{equation}
Then the gradient terms (denoted by $Q_1$) of \eqref{s4:G-evl} at the critical point of $G$ satisfy
\begin{align}\label{s4:Q1}
  Q_1K^{3-\alpha}= & \biggl(-8(\alpha-1)^2\frac{K}{H^2}-2(\alpha-1)(\alpha-2)+(\alpha-1)(\alpha-2)\frac{H^2}K\biggr)\nonumber\\
  &\quad \times\dot{\Psi}^{kl}\nabla_kK\nabla_lK+2HK\ddot{\Psi}^{kl,pq}\nabla_ih_{kl}\nabla^ih_{pq}.
\end{align}
Using \eqref{s4:psi-d} -- \eqref{s4:psi-dd}, we have
\begin{align*}
 \dot{\Psi}^{kl}\nabla_kK\nabla_lK=&~  \frac{\alpha}2K^{\frac{\alpha}2-1}(\kappa_2(\nabla_1K)^2+\kappa_1(\nabla_2K)^2) \\
   \ddot{\Psi}^{kl,pq}\nabla_ih_{kl}\nabla^ih_{pq}=&~\frac{\alpha}2(\frac{\alpha}2-1)K^{\frac{\alpha}2-2}\sum_{i=1}^2(\kappa_2^2(\nabla_ih_{11})^2+\kappa_1^2(\nabla_ih_{22})^2)\\
&\quad +\frac{\alpha^2}2K^{\frac{\alpha}2-1}\sum_{i=1}^2\nabla_ih_{11}\nabla_ih_{22}-\alpha K^{\frac{\alpha}2-1}\sum_{i=1}^2(\nabla_ih_{12})^2\\
=&~\frac{\alpha}2(\frac{\alpha}2-1)K^{\frac{\alpha}2-2}((\nabla_1K)^2+(\nabla_2K)^2)\\
&\quad +\alpha K^{\frac{\alpha}2-1}\biggl(\nabla_1h_{11}\nabla_1h_{22}-(\nabla_1h_{12})^2+\nabla_2h_{11}\nabla_2h_{22}-(\nabla_2h_{12})^2\biggr)
\end{align*}
The equation \eqref{s4:G-grad} implies that $\nabla_ih_{11}$ and $\nabla_ih_{22}$ are linearly dependent, i.e., there exist functions $g_1,g_2$ such that
\begin{equation}\label{s4:G-grad-1}
 g_1\nabla_ih_{11}=g_2\nabla_ih_{22}.
\end{equation}
The functions $g_1,g_2$  can be expressed explicitly as follows:
\begin{align*}
  g_1=&2HK-(4(\alpha-1)K+(2-\alpha)H^2)\kappa_2,\\
  g_2=&-2HK+(4(\alpha-1)K+(2-\alpha)H^2)\kappa_1.
\end{align*}
Without loss of generality, we can assume that both $g_1$ and $g_2$ are not equal to zero at the critical point of $G$. In fact, if $g_1=0$, then
\begin{align}\label{s4:g1}
  0=g_1  = & ~\biggl((\alpha-2)(H^2-4K)+2H\kappa_1-4K\biggr)\kappa_2\nonumber\\
  =&~\biggl((\alpha-2)(\kappa_1-\kappa_2)+2\kappa_1\biggr)(\kappa_1-\kappa_2)\kappa_2.
\end{align}
Since $\alpha\leq 2$, we have $(\alpha-2)(\kappa_1-\kappa_2)+2\kappa_1\geq 2\kappa_1>0$. Thus \eqref{s4:g1} is equivalent to $\kappa_2=\kappa_1$ and we have nothing to prove.

By the equation \eqref{s4:G-grad-1}, we have
\begin{align*}
  (\nabla_1K)^2 =& (\kappa_2+g_2^{-1}g_1\kappa_1)^2(\nabla_1h_{11})^2 =  g_2^{-2}4H^2K^2(H^2-4K) (\nabla_1h_{11})^2,\\
   (\nabla_2K)^2 =& (\kappa_2g_1^{-1}g_2+\kappa_1)^2(\nabla_2h_{22})^2  =  g_1^{-2}4H^2K^2(H^2-4K) (\nabla_2h_{22})^2.
\end{align*}
Using the Codazzi identity, the equation \eqref{s4:G-grad-1} also implies that
\begin{align*}
&\nabla_1h_{11}\nabla_1h_{22}-(\nabla_1h_{12})^2+\nabla_2h_{11}\nabla_2h_{22}-(\nabla_2h_{12})^2\\
=&~\nabla_1h_{11}\nabla_1h_{22}-(\nabla_1h_{22})^2+\nabla_2h_{11}\nabla_2h_{22}-(\nabla_2h_{11})^2\\
=&~ g_2^{-2}g_1(g_2-g_1)(\nabla_1h_{11})^2 +g_1^{-2}g_2(g_1-g_2)(\nabla_2h_{22})^2\\
  = & (2-\alpha)(H^2-4K)\biggl(g_2^{-2}g_1(\nabla_1h_{11})^2+g_1^{-2}g_2(\nabla_2h_{22})^2\biggr).
\end{align*}
Therefore we can write the right hand side of \eqref{s4:Q1} as linear combination of $(\nabla_{1}h_{11})^2$ and $(\nabla_2h_{22})^2$:
 \begin{align}\label{s4:Q1-2}
  \frac 2{\alpha}Q_1K^{4-3\alpha/2}= &  q_{1}g_2^{-2}(\nabla_{1}h_{11})^2+q_2g_1^{-2}(\nabla_2h_{22})^2,
  \end{align}
where the coefficients $q_1,q_2$ satisfy
  \begin{align*}
  q_1=&\biggl(-8(\alpha-1)^2\frac{K^2}{H^2}+2(\alpha-1)(\alpha-2)K+(\alpha-2)H^2\biggr)4H^2K(H^2-4K)\kappa_2\nonumber\\
  &\quad +4(2-\alpha)H^3K^2(H^2-4K)\\
  =& -32(\alpha-1)^2K^3(H^2-4K)\kappa_2-4(2-\alpha)(2\alpha-1)H^2K^2(H^2-4K)\kappa_2\\
  &\quad -4(2-\alpha)H^2K(H^2-4K)\kappa_2^3
\end{align*}
and
 \begin{align*}
  q_2=& -32(\alpha-1)^2K^3(H^2-4K)\kappa_1-4(2-\alpha)(2\alpha-1)H^2K^2(H^2-4K)\kappa_1\\
  &\quad -4(2-\alpha)H^2K(H^2-4K)\kappa_1^3.
\end{align*}
It can be checked directly that $q_1$ and $q_2$ are both non-positive if $\alpha\in [1/2,2]$. Thus the gradient terms $Q_1$ of \eqref{s4:G-evl} are non-positive at a critical point of $G$ if $\alpha\in [1/2,2]$. The maximum principle implies that $\max_{M_t}G$ is non-increasing in time. It follows that $G(x,t)\leq \max_{M_0}G(x,0)$. Since $1<K\leq C$ for some constant $C>0$ by Theorem \ref{s3:thm-2} and Theorem \ref{s4:F-bd}, we have
\begin{equation}\label{s4:H-bd}
  H^2= ~4K+K^{\alpha-2}G\leq ~C.
\end{equation}
Finally, the estimate \eqref{s4:lem-pinch} follows from \eqref{s4:H-bd} and $K>1$ immediately.
\endproof

Now we have proved that the principal curvatures $\kappa_i$ of $M_t$ satisfy the uniform estimate $0<1/C\leq \kappa_i\leq C$ for some constant $C>0$, which is equivalent to the $C^2$ estimate for $M_t$. Since the functions $f$ we considered in Theorem \ref{thm1-1} are inverse-concave, we can apply an argument similar to that in \cite[\S 5]{And-Wei2017-2} to derive higher regularity estimates. The standard continuation argument then implies the long time existence of the flow, and the argument in \cite[\S 6]{And-Wei2017-2} implies the smooth convergence to a geodesic sphere as time goes to infinity.

\section{Horospherically convex regions}\label{sec:h-convex}

In this section we will investigate some of the properties of horospherically convex regions in hyperbolic space (that is, regions which are given by the intersection of a collection of horo-balls).  In particular, for such regions we define a horospherical Gauss map, which is a map to the unit sphere, and we show that each horospherically convex region is completely described in terms of a scalar function on the sphere which we call the \emph{horospherical support function}.  There are interesting formal similarities between this situation and that of convex Euclidean bodies.  For the purposes of this paper the main result we need is that the modified Quermassintegrals are monotone with respect to inclusion for horospherically convex domains.  However we expect that the description of horospherically convex regions which we develop here will be useful in further investigations beyond the scope of this paper.

We remark that a similar development is presented in \cite{EGM}, but in a slightly different context:  In that paper the `horospherically convex' regions are those which are intersections of complements of horo-balls (corresponding to principal curvatures greater than $-1$ everywhere, while we deal with regions which are intersections of horo-balls, corresponding to principal curvatures greater than $1$.  Our condition is more stringent but is more useful for the evolution equations we consider here.

\subsection{The horospherical Gauss map}\label{sec:5-1}

The horospheres in hyperbolic space are the submanifolds with constant principal curvatures equal to $1$ everywhere.
If we identify ${\mathbb H}^{n+1}$ with the future time-like hyperboloid in Minkowski space $\RR^{n+1,1}$, then the condition of constant principal curvatures equal to $1$ implies that the null vector $\bar\e := X-\nu$ is constant on the hypersurface, since we have ${\mathcal W}={\mathrm I}$, and hence
$$
D_v\bar\e = D_v(X-\nu) = DX(({\mathrm I}-{\mathcal W})(v)) = 0
$$
for all tangent vectors $v$.  Then we observe that
$$
X\cdot\bar\e = X\cdot (X-\nu) = -1,
$$
from which it follows that the horosphere is the intersection of the null hyperplane $\{X:\ X\cdot\bar\e=-1\}$ with the hyperboloid ${\mathbb H}^{n+1}$.  The horospheres are therefore in one-to-one correspondence with points $\bar\e$ in the future null cone, which are given by $\{\bar\e=\lambda(\e,1):\ \e\in S^n,\ \lambda>0\}$, and there is a one-parameter family of these for each $\e\in S^n$.  For convenience we parametrise these by their signed geodesic distance $s$ from the `north pole' $N=(0,1)\in{\mathbb H}^{n+1}$, satisfying $-1 = \lambda(\cosh(s)N+\sinh(s)(\e,0))\cdot (\e,1) = -\lambda \E^{s}$.  It follows that $\lambda = \E^{-s}$.  Thus we denote by $H_{\e}(s)$ the horosphere
$\{X\in{\mathbb H}^{n+1}:\ X\cdot (\e,1) = -\E^s\}$.  The interior region (called a \emph{horo-ball}) is denoted by
$$
B_{\e}(s) = \{X\in{\mathbb H}^{n+1}:\ 0>X\cdot(\e,1) >-\E^s\}.
$$


A region $\Omega$ in ${\mathbb H}^{n+1}$ is \emph{horospherically convex} (or h-convex for convenience) if every boundary point $p$ of $\partial\Omega$ has a supporting horo-ball, i.e. a horo-ball $B$ such that  $\Omega\subset B$ and $p\in\partial B$.  If the boundary of $\Omega$ is a smooth hypersurface, then this implies that every principal curvature of $\partial\Omega$ is greater than or equal to $1$ at $p$.  We say that $\Omega$ is \emph{uniformly h-convex} if there is $\delta>0$ such that all principal curvatures exceed $1+\delta$.

Let $M^n=\partial\Omega$ be a hypersurface which is at the boundary of a horospherically convex region $\Omega$.   Then the horospherical Gauss map ${\mathbf e}:\ M\to S^n$ assigns to each $p\in M$ the point $\e(p) =\pi(X(p)-\nu(p))\in S^n$, where $\pi(x,y) = \frac{x}{y}$ is the radial projection from the future null cone onto the sphere $S^n\times\{1\}$.  We observe that the derivative of ${\mathbf e}$ is non-singular if $M$ is uniformly h-convex:  If $v$ is a tangent vector to $M$, then
$$
D\e(v) = D\pi\big|_{X-\nu}\left(({\mathcal W}-{\textrm I})(v)\right).
$$
Here $\tilde v=({\mathcal W}-{\textrm I})(v)$ is a non-zero tangent vector to $M$ since the eigenvalues $\kappa_i$ of ${\mathcal W}$ are greater than $1$.  In particular $\tilde v$ is spacelike.  On the other hand the kernel of $D\pi|_{X-\nu}$ is the line $\RR(X-\nu)$ consisting of null vectors.  Therefore $D\pi(\tilde v)\neq 0$.  Thus $D\e$ is an injective linear map, hence an isomorphism.  It follows that $\e$ is a diffeomorphism from $M$ to $S^n$.

\subsection{The horospherical support function}

Let $M^n=\partial\Omega$ be the boundary of a compact h-convex region.  Then for each $\e\in S^n$ we define the \emph{horospherical support function} of $\Omega$ (or $M$) in direction $\e$ by
$$
u(\e):= \inf\{s\in\RR:\ \Omega\subset B_{\e}(s)\}.
$$


Alternatively, define $f_\e:\ {\mathbb H}^{n+1}\to \RR$ by $f_\e(\xi) = \log\left(-\xi\cdot(\e,1)\right)$.  This is a smooth function on ${\mathbb H}^{n+1}$, and we have the alternative characterisation
\begin{equation}\label{eq:defu}
u(\e) = \sup\{f_\e(\xi):\ \xi\in\Omega\}.
\end{equation}
The function $u$ is called the \emph{horospherical support function} of the region $\Omega$, and $B_\e(u(\e))$ is the \emph{supporting horo-ball} in direction $\e$.  The support function completely determines a horospherically convex region $\Omega$, as an intersection of horo-balls:
\begin{equation}\label{eq:utoOmega}
\Omega = \bigcap_{\e\in S^n}B_\e(u(\e)).
\end{equation}

\subsection{Recovering the region from the support function}

If the region is uniformly h-convex, in the sense that all principal curvatures are greater than $1$, then there is a unique point of $M$ in the boundary of the supporting horo-ball $B_\e(u(\e))$.  We denote this point by $\bar X(\e)$.  We observe that $\bar X = X\circ \e^{-1}$, so if $M$ is smooth and uniformly h-convex (so that $\e$ is a diffeomorphism) then $\bar X$ is a smooth embedding.

We will show that $\bar X$ can be written in terms of the support function $u$, as follows:   Choose local coordinates $\{x^i\}$ for $S^n$ near $\e$.  We write $\bar X(\e)$ as a linear combination of the basis consisting of the two null elements $(\e,1)$ and $(-\e,1)$, together with $(\e_j,0)$, where $\e_j = \frac{\partial\e}{\partial x^j}$ for $j=1,\cdots,n$:
$$
\bar X(\e) = \alpha (-\e,1) + \beta (\e,1) + \gamma^j(\e_j,0)
$$
for some coefficients, $\alpha$, $\beta$, $\gamma^j$.
Since $\bar X(\e)\in{\mathbb H}^{n+1}$ we have $|\gamma|^2-4\alpha\beta=-1$, so that $\beta = \frac{1+|\gamma|^2}{4\alpha}$.  We also know that $\bar X(\e)\cdot (\e,1) = -\E^{u(\e)}$ since $\bar X(\e)\in H_\e(u(\e))$, implying that $\alpha = \frac12\E^u$.  This gives
$$
\bar X(\e) = \frac12\E^{u(\e)}(-\e,1)+\frac12\E^{-u(\e)}(1+|\gamma|^2)(\e,1)+\gamma^j(\e_j,0).
$$

Furthermore,  the normal to $M$ at the point $\bar X(\e)$ must coincide with the normal to the horosphere $H_\e(u(\e))$, which is given by
\begin{equation}\label{eq:nue2}
\nu = \bar X-\bar e = \bar X-\E^{-u(\e)}(\e,1).
\end{equation}
Since $|\bar X|^2=-1$ we have $\partial_j\bar X\cdot\bar X=0$, and hence
\begin{align*}
0 &= \partial_j\bar X\cdot \nu\\
&= \partial_j\bar X\cdot\left(\bar X-\E^{-u}(\e,1)\right)\\
&=-\E^{-u}\partial_jX\cdot (\e,1).
\end{align*}
Observing that $(\e,1)\cdot(\e,1)=0$ and $(\e_i,0)\cdot(\e,1)=0$, and that $\partial_j\e_i = -\bar g_{ij}\e$ and $\partial_j\e=\e_j$, the condition becomes
\begin{align*}
0 &= \partial_jX\cdot(\e,1)\\
&= \left(\frac12\E^uu_j(-\e,1)-\gamma_j(\e,0)\right)\cdot(\e,1)\\
&=-\E^uu_j-\gamma_j,
\end{align*}
where $\gamma_j = \gamma^i\bar g_{ij}$ and $\bar g$ is the standard metric on $S^n$.
It follows that we must have $\gamma_j = -\E^uu_j$.  This gives the following expression for $\bar X$:
\begin{align}\label{eq:barX2}
\bar X(\e) &= \left(-\E^u\bar\nabla u+\left(\frac12\E^u|\bar\nabla u|^2-\sinh u\right)\e,\frac12\E^u|\bar\nabla u|^2+\cosh u\right)\\
&=-\E^uu_p\bar g^{pg}(\e_q,0)+\frac12\left(\E^u|\bar\nabla u|^2+\E^{-u}\right)(\e,1)+\frac12\E^u(-\e,1).\label{eq:barX3}
\end{align}

\subsection{A condition for horospherical convexity}\label{sec:5-4}

Given a smooth function $u$, we can use the expression \eqref{eq:barX2} to define a map to hyperbolic space.  In this section we determine when the resulting map is an embedding defining a horospherically convex hypersurface.

If order for $\bar X$ to be an immersion, we require the derivatives $\partial_j\bar X$ to be linearly independent.  Since we have constructed $\bar X$ in such a way that $\partial_jX$ is orthogonal to the normal vector $\nu$ to the horosphere $B_{\e}(u(\e))$,  $\partial_j\bar X$ is a linear combination of the basis for the space orthogonal to $\nu$ and $\bar X$ given by the projections $E_k$ of $(e_k,0)$, $k=1,\cdots,n$.  Computing explicity, we find
\begin{equation}\label{eq:Ek}
E_k = (\e_k,0)-u_k(\e,1).
\end{equation}
The immersion condition is therefore equivalent to invertibility of the matrix $A$ define by
$$
A_{jk} = -\partial_j\bar X\cdot E_k.
$$
Given that $A$ is non-singular, we have that $\bar X$ is an immersion with unit normal vector $\nu(\e)$, and we can differentiate the equation $X-\nu = \E^{-u}(\e,1)$ to obtain the following:
$$
-(h_j^p-\delta_j^p)\partial_pX = -u_j\E^{-u}(\e,1)+\E^{-u}(\e_j,0).
$$
Taking the inner product with $E_k$ using \eqref{eq:Ek}, we obtain
\begin{equation}\label{eq:Avs2ff}
(h_j^p-\delta_j^p)A_{pk} = \E^{-u}\bar g_{jk}.
\end{equation}
It follows that $A$ is non-singular precisely when ${\mathcal W}-\mathrm{I}$ is non-singular, and is given by
\begin{equation}\label{eq:A-vs-W}
A_{jk} = \E^{-u}\left[\left({\mathcal W}-\mathrm{I}\right)^{-1}\right]_j^p\bar g_{pk}.
\end{equation}
In particular, $A$ is symmetric, and ${\mathcal W}-\mathrm{I}$ is positive definite (corresponding to uniform h-convexity) if and only if the matrix $A$ is positive definite.  We conclude that if $u$ is a smooth function on $S^n$, then the map $X$ defines an embedding to the boundary of a uniformly h-convex region if and only if the tensor $A$ computed from $u$ is positive definite.

Computing $A$ explicitly using \eqref{eq:barX3}, we obtain
\begin{align*}
A_{jk} &= \left((\bar\nabla_j(\E^u\bar\nabla u),0)-\E^uu_j(\e,0) -\frac12\partial_j(\E^u|\bar\nabla u|^2+\E^{-u})(\e,1)\right.\\
&\quad\quad\null\left.-\frac12\E^uu_j(-\e,1)-\left(\frac12\E^u|\nabla u|^2-\sinh u\right)(\e_j,0)\right)\cdot((\e_k,0)-u_k(\e,1))\\
&=\bar\nabla_j(\E^u\bar\nabla_ku)-\frac12\E^u|\bar\nabla u|^2\bar g_{jk}+\sinh u\bar g_{jk}.
\end{align*}
It is convenient to write this in terms of the function $\varphi=\E^u$:
\begin{equation}\label{eq:Ainphi}
A_{jk} = \bar\nabla_j\bar\nabla_k\varphi-\frac{|\bar\nabla\varphi|^2}{2\varphi}\bar g_{jk}+\frac{\varphi-\varphi^{-1}}{2}\bar g_{jk}.
\end{equation}

\subsection{Monotonicity of the modified Quermassintegrals}

We will prove that the modified quermassintegrals $\tilde W_k$ are monotone with respect to inclusion by making use of the following result:

\begin{prop}\label{prop:connect-h-convex}
Suppose that $\Omega_1\subset\Omega_2$ are smooth, strictly h-convex domains in ${\mathbb H}^{n+1}$.  Then there exists a smooth map $X:\ S^n\times[0,1]\to{\mathbb H}^{n+1}$ such that
\begin{enumerate}
\item $X(.,t)$ is a uniformly h-convex embedding of $S^n$ for each $t$;
\item $X(S^n,0)=\partial\Omega_0$ and $X(S^n,1)=\partial\Omega_1$;
\item The hypersurfaces $M_t=X(S^n,t)$ are expanding, in the sense that $\frac{\partial X}{\partial t}\cdot \nu\geq 0$. Equivalently, the enclosed regions $\Omega_t$ are nested:  $\Omega_s\subset\Omega_t$ for each $s\leq t$ in $[0,1]$.
\end{enumerate}
\end{prop}

\begin{proof}
Let $u_0$ and $u_1$ be the horospherical support functions of $\Omega_0$ and $\Omega_1$ respectively,  The inclusion $\Omega_0\subset\Omega_1$ implies that $u_0(\e)\leq u_1(\e)$ for all $\e\in S^n$, by the characterisation \eqref{eq:defu}.

We define $X(\e,t) = \bar X[u(\e,t)]$ according to the formula \eqref{eq:barX2}, where
$$
\E^{u(\e,t)}=\varphi(\e,t) := (1-t)\varphi_0(\e)+t\varphi_1(\e),
$$
where $\varphi_i = \E^{u_i}$ for $i=0,1$.  Then $u(\e,t)$ is increasing in $t$, and it follows that the regions $\Omega_t$ are nested, by the expression \eqref{eq:utoOmega}.

We check that each $\Omega_t$ is a strictly h-convex region, by showing that the matrix $A$ constructed from $u(.,t)$ is positive definite for each $t$:  We have
\begin{align*}
A_{jk}[u(.,t)] &= \bar\nabla_j\bar\nabla_k\varphi_t -\frac{|\bar\nabla\varphi_t|^2}{2\varphi_t}\bar g_{jk}+\frac{\varphi_t-\varphi_t^{-1}}{2}\bar g_{jk}\\
&= (1-t)A_{jk}[u_0]+tA_{jk}[u_1]\\
&\quad\null +\frac12\left(-\frac{|(1-t)\bar\nabla\varphi_0+t\bar\nabla\varphi_1|^2}{(1-t)\varphi_0+t\varphi_1}+(1-t)\frac{|\bar\nabla\varphi_0|^2}{\varphi_0}+t\frac{|\bar\nabla\varphi_1|^2}{\varphi_1}\right)\bar g_{jk}\\
&\quad\null +\frac12\left(-\frac{1}{(1-t)\varphi_0+t\varphi_1}+\frac{1-t}{\varphi_0}+\frac{t}{\varphi_1}\right)\bar g_{jk}\\
&=(1-t)A_{jk}[u_0]+tA_{jk}[u_1]+t(1-t)\frac{|\varphi_0\bar\nabla\varphi_1-\varphi_1\bar\nabla\varphi_0|^2+|\varphi_1-\varphi_0|^2}{2\varphi_0\varphi_1((1-t)\varphi_0+t\varphi_1)}\bar g_{jk}\\
&\geq (1-t)A_{jk}[u_0]+tA_{jk}[u_1].
\end{align*}
Since $A_{jk}[u_0]$ and $A_{jk}[u_1]$ are positive definite, so is $A_{jk}[u_t]$ for each $t\in[0,1]$, and we conclude that the region $\Omega_t$ is uniformly h-convex.
\end{proof}

\begin{cor}\label{s5:cor}
The modified quermassintegral ${\widetilde W}_k$ is monotone with respect to inclusion for h-convex domains:  That is, if $\Omega_0$ and $\Omega_1$ are h-convex domains with $\Omega_0\subset\Omega_1$, then $\widetilde W_k(\Omega_0)\leq\widetilde W_k(\Omega_1)$.
\end{cor}

\begin{proof}
We use the map $X$ constructed in the Proposition \ref{prop:connect-h-convex}.  By Lemma \ref{s2:lem3} we have
$$
\frac{d}{dt}\widetilde W_k(\Omega_t) = \int_{M_t}E_k(\lambda)\frac{\partial X}{\partial t}\cdot \nu\,d\mu_t.
$$
Since each $M_t$ is h-convex, we have $\lambda_i>0$ and hence $E_k(\lambda)>0$, and from Proposition \ref{prop:connect-h-convex} we have $\frac{\partial X}{\partial t}\cdot \nu\geq 0$.  It follows that $\frac{d}{dt}\widetilde W_k(\Omega_t)\geq 0$ for each $t$, and hence $\widetilde W_k(\Omega_0)\leq \widetilde W_k(\Omega_1)$ as claimed.
\end{proof}

\subsection{Evolution of the horospherical support function}
We end this section with the following observation that the flow \eqref{flow-VMCF-2} of h-convex hypersurfaces is equivalent to an initial value problem for the horospherical support function.
\begin{prop}
The flow \eqref{flow-VMCF-2} of h-convex hypersurfaces in $\mathbb{H}^{n+1}$ is equivalent to the following initial value problem
\begin{equation}\label{s5:flow-u}
 \left\{\begin{aligned}
 \frac{\partial}{\partial t}\varphi=&~-F((A_{ij})^{-1})+\varphi\phi(t),\\
 \varphi(\cdot,0)=&~\varphi_0(\cdot)
  \end{aligned}\right.
 \end{equation}
on $S^n\times [0,T)$, where $\varphi=e^u$ and $A_{ij}$ is the matrix defined in \eqref{eq:Ainphi}.
\end{prop}
\proof
Suppose that $X(\cdot,t):M\to \mathbb{H}^{n+1}, t\in [0,T)$ is a family of smooth, closed and strictly h-convex hypersurfaces satisfying the flow \eqref{flow-VMCF-2}. Then as explained in \S \ref{sec:5-1}, the horospherical Gauss map $\mathbf{e}$ is a diffeomorphism from $M_t=X(M,t)$ to $S^n$. We can reparametrize $M_t$ such that $\bar{X}=X\circ \mathbf{e}^{-1}$ is a family of smooth embeddings from $S^n$ to $\mathbb{H}^{n+1}$. Then
\begin{equation*}
  \frac{\partial}{\partial t}\bar{X}(z,t)= \frac{\partial}{\partial t}{X}(p,t)+ \frac{\partial X}{\partial p_i} \frac{\partial p^i}{\partial t},
\end{equation*}
where $z\in S^n$ and $p=\mathbf{e}^{-1}(z)\in M_t$. Since $\frac{\partial X}{\partial p_i}$ is tangent to $M_t$, we have
\begin{equation}\label{s5:u-1}
  \frac{\partial}{\partial t}\bar{X}(z,t)\cdot \nu(z,t)= \frac{\partial}{\partial t}{X}(p,t)\cdot\nu(z,t)=\phi(t)-F(\mathcal{W}-I).
\end{equation}
On the other hand, by \eqref{eq:nue2} we have
\begin{equation}\label{s5:u-1-1}
  \bar{X}(z,t)-\nu(z,t)=e^{-u(z,t)}(z,1),
\end{equation}
where $u(\cdot,t)$ is the horospherical support function of $M_t$ and $(z,1)\in \mathbb{R}^{n+1,1}$ is a null vector. Differentiating \eqref{s5:u-1-1} in time gives that
 \begin{equation*}
    \frac{\partial}{\partial t}\bar{X}(z,t)- \frac{\partial}{\partial t}\nu(z,t)=-e^{-u(z,t)} \frac{\partial u}{\partial t}(z,1).
 \end{equation*}
Then
\begin{align}\label{s5:u-2}
  \frac{\partial}{\partial t}\bar{X}(z,t)\cdot\nu(z,t) =& -e^{-u(z,t)} \frac{\partial u}{\partial t}(z,1)\cdot \nu(z,t) \nonumber\\
   =& -e^{-u(z,t)} \frac{\partial u}{\partial t}(z,1)\cdot (\bar{X}(z,t)-e^{-u(z,t)}(z,1)) \nonumber\\
   =&-e^{-u(z,t)} \frac{\partial u}{\partial t}(z,1)\cdot \frac 12e^{u(z,t)}(-z,1) \nonumber\\
   =& \frac{\partial u}{\partial t},
\end{align}
where we used \eqref{eq:nue2} and \eqref{eq:barX3}. Combining \eqref{s5:u-1} and \eqref{s5:u-2} implies that
\begin{equation}\label{s5:u-3}
  \frac{\partial u}{\partial t}=\phi(t)-F(\mathcal{W}-I).
\end{equation}
Therefore $\varphi=e^u$ satisfies
\begin{align}
  \frac{\partial \varphi}{\partial t}=&e^u\phi(t)-F(e^u(\mathcal{W}-I))\nonumber\\
  =&\varphi\phi(t)-F((A_{ij})^{-1})
\end{align}
with $A_{ij}$ defined as in \eqref{eq:Ainphi}.

Conversely, suppose that we have a smooth solution $\varphi(\cdot,t)$ of the initial value problem \eqref{s5:flow-u} with $A_{ij}$ positive definite. Then by the discussion in \S \ref{sec:5-4}, the map $\bar{X}$ given in \eqref{eq:barX3} using the function $u=\log \varphi$ defines a family of smooth h-convex hypersurfaces in $\mathbb{H}^{n+1}$. We claim that we can find a family of diffeomorphisms $\xi(\cdot,t): S^n\to S^n$ such that $X(z,t)=\bar{X}(\xi(z,t),t)$ solves the flow equation \eqref{flow-VMCF-2}. Since
\begin{align*}
  \frac{\partial}{\partial t}{X}(z,t)=& \frac{\partial}{\partial t}\bar{X}(\xi,t)+\partial_i \bar{X} \frac{\partial \xi^i}{\partial t}\\
  =&(\frac{\partial}{\partial t}\bar{X}(\xi,t)\cdot \nu(\xi,t))\nu(\xi,t)+(\frac{\partial}{\partial t}\bar{X}(\xi,t))^{\top}+\partial_i \bar{X} \frac{\partial \xi^i}{\partial t}\\
  =& (\phi(t)-F(\mathcal{W}-I))\nu(\xi,t)+(\frac{\partial}{\partial t}\bar{X}(\xi,t))^{\top}+\partial_i \bar{X} \frac{\partial \xi^i}{\partial t},
\end{align*}
where $(\cdot)^{\top}$ denotes the tangential part, it suffices to find a family of diffeomorphisms $\xi:S^n\to S^n$ such that
\begin{equation*}
  (\frac{\partial}{\partial t}\bar{X}(\xi,t))^{\top}+\partial_i \bar{X} \frac{\partial \xi^i}{\partial t}=0,
\end{equation*}
which is equivalent to
\begin{equation}\label{s5:u-4}
  (\frac{\partial}{\partial t}\bar{X}(\xi,t))^{\top}\cdot E_j-A_{ij}\frac{\partial \xi^i}{\partial t}=0.
\end{equation}
By assumption $A_{ij}$ is positive definite on $S^n\times [0,T)$, the standard theory of the ordinary differential equations implies that the system \eqref{s5:u-4} has a unique smooth solution for the initial condition $\xi(z,0)=z$. This completes the proof.
\endproof

\section{Proof of Theorem \ref{thm1-5}}\label{sec:thm5-pf}
In this section, we will give the proof of Theorem \ref{thm1-5}.

\subsection{Pinching estimate}\label{sec:Pinch}

Firstly, we prove the following pinching estimate for the shifted principal curvatures of the evolving hypersurfaces along the flow \eqref{flow-VMCF-2}.

%
\begin{prop}
Let $M_t$ be a smooth solution to the flow \eqref{flow-VMCF-2} on $[0,T)$ and assume that $F$ satisfies the assumption in Theorem \ref{thm1-5}. Then there exists a constant $C>0$ depending only on $M_0$ such that
\begin{equation}\label{s5:pinc-1}
  \lambda_n~\leq~C\lambda_1
\end{equation}
for all $t\in [0,T)$, where $\lambda_n=\kappa_n-1$ is the largest shifted principal curvature  and $\lambda_1=\kappa_1-1$ is the smallest shifted principal curvature.
\end{prop}
\proof
We consider the four cases of $F$ separately.

(i). $F$ is concave and $F$ vanishes on the boundary of the positive cone $\Gamma_+$. Define a function $G=F^{-1}\mathrm{tr}(S)$ on $M\times [0,T)$. Then the equations \eqref{s2:evl-F} and \eqref{s2:evl-S} imply that
\begin{align}\label{s5:G1}
  \frac{\partial}{\partial t}G =& ~F^{-1}\frac{\partial}{\partial t}\mathrm{tr}(S)-F^{-2}\mathrm{tr}(S)\frac{\partial}{\partial t}F \displaybreak[0]\nonumber\\
  = & ~\dot{F}^{kl}\nabla_k\nabla_lG+2F^{-1}\dot{F}^{kl}\nabla_kF\nabla_lG+F^{-1}\sum_{i=1}^n\ddot{F}^{kl,pq}\nabla_ih_{kl}\nabla^ih_{pq}\displaybreak[0]\nonumber\\
  &\quad +\phi(t)f^{-2}\left(\mathrm{tr}(S)\sum_k\dot{f}^k\lambda_k^2-f|S|^2\right)+f^{-1}\left(n\sum_k\dot{f}^k\lambda_k^2-|S|^2\sum_k\dot{f}^k\right).
\end{align}
Since $F$ is concave, by the inequality \eqref{s2:f-conc} we have
\begin{align*}
  \mathrm{tr}(S)\sum_k\dot{f}^k\lambda_k^2-f|S|^2=& \sum_{k,l}\left(\dot{f}^k\lambda_k^2\lambda_l-\dot{f}^k\lambda_k\lambda_l^2\right) \\
  = &~\frac 12\sum_{k,l}(\dot{f}^k-\dot{f}^l)(\lambda_k-\lambda_l)\lambda_k\lambda_l~\leq~0,
\end{align*}
and
\begin{align*}
  n\sum_k\dot{f}^k\lambda_k^2-|S|^2\sum_k\dot{f}^k= & \sum_{k,l}\left(\dot{f}^k\lambda_k^2-\dot{f}^k\lambda_l^2\right) \\
  = &~\frac 12\sum_{k,l}(\dot{f}^k-\dot{f}^l)(\lambda_k^2-\lambda_l^2)~\leq~0.
\end{align*}
Thus the zero order terms of \eqref{s5:G1} are always non-positive. The concavity of $F$ also implies that the third term of \eqref{s5:G1} is non-positive. Then we have
\begin{align}\label{s5:G2}
  \frac{\partial}{\partial t}G \leq& ~\dot{F}^{kl}\nabla_k\nabla_lG+2F^{-1}\dot{F}^{kl}\nabla_kF\nabla_lG.
\end{align}
The maximum principle implies that the supremum of $G$ over $M_t$ is decreasing in time along the flow \eqref{flow-VMCF-2}. The assumption that $f$ approaches zero on the boundary of the positive cone $\Gamma_+$ then guarantees that the region $\{G(t)\leq \sup_{t=0}G\}\subset \Gamma_+$ does not touch the boundary of $\Gamma_+$. Since $G$ is homogeneous of degree zero with respect to $\lambda_i$, this implies that $\lambda_n\leq C\lambda_1$ for some constant $C>0$ depending only on $M_0$ for all $t\in [0,T)$.

(ii). $F$ is concave and inverse concave. Define a tensor $T_{ij}=S_{ij}-\varepsilon~ \mathrm{tr}(S)\delta_i^j$, where $\varepsilon$ is chosen such that $T_{ij}$ is positive definite initially. Clearly, $0<\varepsilon\leq \frac 1n$. The evolution equation \eqref{s2:evl-S} implies that
\begin{align}\label{s5:evl-T}
 \frac{\partial}{\partial t}T_{ij}  =&~ \dot{F}^{kl}\nabla_k\nabla_lT_{ij}+\ddot{F}^{kl,pq}\nabla_ih_{kl}\nabla_jh_{pq}-\varepsilon\left(\sum_{i=1}^n\ddot{F}^{kl,pq}\nabla_ih_{kl}\nabla^ih_{pq}\right)\delta_i^j\nonumber\\  &\quad +\left(\sum_{k=1}^n\dot{f}^k\lambda_k^2+2f-2\phi(t)\right)T_{ij} -\left(\phi(t)+\sum_{k=1}^n\dot{f}^k\right)\left(T_i^kT_{kj}+2\varepsilon \mathrm{tr}(S)T_{ij}\right)\nonumber\displaybreak[0]\\
   &\quad +\varepsilon \left(\phi(t)+\sum_{k=1}^n\dot{f}^k\right)\left(|S|^2-\varepsilon (\mathrm{tr}(S))^2\right)\delta_i^j+\sum_{k=1}^n\dot{f}^k\lambda_k^2(1-\varepsilon n)\delta_i^j.
\end{align}
We will apply the tensor maximum principle in Theorem \ref{s2:tensor-mp} to show that $T_{ij}$ is positive definite for $t>0$. If not, there exists a first time $t_0>0$ and some point $x_0\in M_{t_0}$ such that $T_{ij}$ has a null vector $v\in T_{x_0}M_{t_0}$, i.e., $T_{ij}v^j=0$ at $(x_0,t_0)$. The second line of \eqref{s5:evl-T} satisfies the null vector condition and can be ignored.  The last line of \eqref{s5:evl-T} is also nonnegative, since $0<\varepsilon<\frac 1n$ and $|S|^2\geq (\mathrm{tr}(S))^2/n$. For the gradient terms in \eqref{s5:evl-T}, Theorem 4.1 of \cite{And2007} implies that
\begin{align*}
  &\ddot{F}^{kl,pq}\nabla_ih_{kl}\nabla_jh_{pq}v^iv^j-\varepsilon\left(\sum_{i=1}^n\ddot{F}^{kl,pq}\nabla_ih_{kl}\nabla^ih_{pq}\right)|v|^2\displaybreak[0]\\ &+\sup_{\Lambda}2a^{kl}\left(2\Lambda_k^p\nabla_lT_{ip}v^i-\Lambda_k^p\Lambda_l^qT_{pq}\right)\geq 0
\end{align*}
for the null vector $v$ provided that $F$ is concave and inverse concave. Thus by Theorem \ref{s2:tensor-mp}, the tensor $T_{ij}$ is positive definite for $t\in [0,T)$. Equivalently,
\begin{equation*}
  \lambda_1~\geq~\varepsilon (\lambda_1+\cdots+\lambda_n)
\end{equation*}
for any $t\in [0,T)$, which implies the pinching estimate \eqref{s5:pinc-1}.

(iii). $F$ is inverse concave and $F_*$ approaches zero on the boundary of $\Gamma_+$.  In this case, we define $T_{ij}=S_{ij}-\varepsilon F\delta_i^j$, where $\varepsilon$ is chosen such that $T_{ij}$ is positive definite initially. By \eqref{s2:evl-F} and \eqref{s2:evl-S},
\begin{align}\label{s5:evl-T-i}
 \frac{\partial}{\partial t}T_{ij}  =&~ \dot{F}^{kl}\nabla_k\nabla_lT_{ij}+\ddot{F}^{kl,pq}\nabla_ih_{kl}\nabla^jh_{pq} +(\dot{f}^{k}\lambda_k^2+2f-2\phi(t))S_{ij}\nonumber\\
   &\quad -(\phi(t)+\sum_{k=1}^n\dot{f}^{k}) S_{ik}S_{kj}+\dot{f}^{k}\lambda_k^2\delta_i^j-\varepsilon (F-\phi(t))\dot{f}^{k}\lambda_k(\lambda_k+2)\delta_i^j.
\end{align}
Suppose $v=e_1$ is the null eigenvector of $T_{ij}$ at $(x_0,t_0)$ for some first time $t_0>0$. Denote the zero order terms of \eqref{s5:evl-T-i} by $Q_0$. At the point $(x_0,t_0)$, $\varepsilon F$ is the smallest eigenvalue of $S_{ij}$ with corresponding eigenvector $v$. Then
\begin{align*}
  Q_0v^iv^j =& (\dot{f}^k\lambda_k^2+2f-2\phi(t))\varepsilon f|v|^2 +(f-\phi(t)-\dot{f}^k\kappa_k)\varepsilon^2f^2|v|^2  \nonumber\\
   & +\dot{f}^k\lambda_k^2|v|^2-\varepsilon (f-\phi(t))\dot{f}^k\lambda_k(\lambda_k+2)|v|^2\displaybreak[0]\nonumber\\
   =&\dot{f}^k\lambda_k^2(1+\varepsilon \phi(t))|v|^2-\varepsilon^2f^2(\sum_k\dot{f}^k+\phi(t))|v|^2\displaybreak[0]\nonumber\\
   =&|v|^2\biggl(\dot{f}^k\lambda_k\varepsilon(\lambda_k-\varepsilon f)\phi(t)+\sum_k\dot{f}^k(\lambda_k^2-\varepsilon^2f^2)\biggr)\geq ~0.
\end{align*}
By Theorem \ref{s2:tensor-mp}, to show that $T_{ij}$ remains positive definite for $t>0$, it suffices to show that
\begin{align*}
  Q_1: =& \ddot{F}^{kl,pq}\nabla_1h_{kl}\nabla_1h_{pq} +2\sup_{\Lambda}\dot{F}^{kl}\left(2\Lambda_k^p\nabla_lT_{1p}-\Lambda_k^p\Lambda_l^qT_{pq}\right)~\geq~0.
\end{align*}
Note that $T_{11}=0$ and $\nabla_kT_{11}=0$ at $(x_0,t_0)$, the supremum over $\Lambda$ can be computed exactly as follows:
\begin{align*}
  2\dot{F}^{kl}\left(2\Lambda_k^p\nabla_lT_{1p}-\Lambda_k^p\Lambda_l^qT_{pq}\right)  =&2\sum_{k=1}^n\sum_{p=2}^n\dot{f}^k\left(2\Lambda_k^p\nabla_kT_{1p}-(\Lambda_k^p)^2T_{pp}\right)\\
  =&2 \sum_{k=1}^n\sum_{p=2}^n\dot{f}^k\left(\frac{(\nabla_kT_{1p})^2}{T_{pp}}-\left(\Lambda_k^p-\frac{\nabla_kT_{1p}}{T_{pp}}\right)^2T_{pp}\right).
\end{align*}
It follows that the supremum is obtained by choosing $\Lambda_k^p=\frac{\nabla_kT_{1p}}{T_{pp}}$. The required inequality for $Q_1$ becomes:
\begin{align*}
  Q_1=& \ddot{F}^{kl,pq}\nabla_1h_{kl}\nabla_1h_{pq} +2 \sum_{k=1}^n\sum_{p=2}^n\dot{f}^k\frac{(\nabla_kT_{1p})^2}{T_{pp}}~\geq~0.
\end{align*}
Using \eqref{s2:F-ddt} to express the second derivatives of $F$ and noting that $\nabla_kT_{1p}=\nabla_kh_{1p}-\varepsilon \nabla_kF\delta_1^p=\nabla_kh_{1p}$ at $(x_0,t_0)$ for $p\neq 1$, we have
\begin{align}\label{s5:Q1-0}
  Q_1=& \ddot{f}^{kl}\nabla_1h_{kk}\nabla_1h_{ll}+2\sum_{k>l}\frac{\dot{f}^k-\dot{f}^l}{\lambda_k-\lambda_l}(\nabla_1h_{kl})^2+2 \sum_{k=1}^n\sum_{l=2}^n\frac{\dot{f}^k}{\lambda_{l}-\varepsilon F}(\nabla_1h_{kl})^2.
\end{align}
Since $f$ is inverse concave, the inequality \eqref{s2:f-invcon-1} implies that the first term of \eqref{s5:Q1-0} satisfies
\begin{align*}
 \ddot{f}^{kl}\nabla_1h_{kk}\nabla_1h_{ll} \geq &  ~ 2f^{-1}(\sum_{k=1}^n\dot{f}^k\nabla_1h_{kk})^2-2\sum_k\frac{\dot{f}^k}{\lambda_k}(\nabla_1h_{kk})^2\displaybreak[0]\\
 =& ~ 2f^{-1}(\nabla_1F)^2-2\sum_k\frac{\dot{f}^k}{\lambda_k}(\nabla_1h_{kk})^2.
\end{align*}
Then
\begin{align*}
 Q_1\geq &~2f^{-1}(\nabla_1F)^2-2\sum_k\frac{\dot{f}^k}{\lambda_k}(\nabla_1h_{kk})^2\\
  &\quad +2\sum_{k>l}\frac{\dot{f}^k-\dot{f}^l}{\lambda_k-\lambda_l}(\nabla_1h_{kl})^2+2 \sum_{k=1}^n\sum_{l=2}^n\frac{\dot{f}^k}{\lambda_{l}-\varepsilon F}(\nabla_1h_{kl})^2\displaybreak[0]\\
  \geq &~2f^{-1}(\nabla_1F)^2-2\frac{\dot{f}^1}{\lambda_1}(\nabla_1h_{11})^2-2\sum_{k>1}\frac{\dot{f}^k}{\lambda_k}(\nabla_1h_{kk})^2\\
  &+2\sum_{k>1}\frac{\dot{f}^k-\dot{f}^1}{\lambda_k-\lambda_1}(\nabla_kh_{11})^2 -2\sum_{k\neq l>1}\frac{\dot{f}^k}{\lambda_l}(\nabla_1h_{kl})^2\displaybreak[0]\\
  &+2\sum_{k>1}\frac{\dot{f}^1}{\lambda_{k}-\varepsilon F}(\nabla_kh_{11})^2+2\sum_{k>1,l>1}\frac{\dot{f}^k}{\lambda_{l}-\varepsilon F}(\nabla_1h_{kl})^2\displaybreak[0]\\
= &~2f^{-1}(\nabla_1F)^2-2\frac{\dot{f}^1}{\lambda_1}(\nabla_1h_{11})^2+2\sum_{k>1}\frac{\dot{f}^k}{\lambda_k-\lambda_1}(\nabla_kh_{11})^2\displaybreak[0]\\
&\quad +2\sum_{k>1,l>1}\dot{f}^k\left(\frac{1}{\lambda_l-\varepsilon F}-\frac 1{\lambda_l}\right)(\nabla_1h_{kl})^2\displaybreak[0]\\
 \geq &~2\left(\frac 1{\varepsilon^2F}-\frac{\dot{f}^1}{\lambda_1}\right)(\nabla_1h_{11})^2\displaybreak[0]\\
 =& ~2\left(\frac {\sum_{k=1}^n\dot{f}^k\lambda_k}{\varepsilon^2F^2}-\frac{\dot{f}^1}{\lambda_1}\right)(\nabla_1h_{11})^2 ~\geq ~0,
  \end{align*}
where we used $\lambda_1=\varepsilon F$ and $\nabla_kh_{11}=\varepsilon \nabla_kF$ at $(x_0,t_0)$, and the inequality in \eqref{s2:f-invcon} due to the inverse concavity of $f$. Theorem \ref{s2:tensor-mp} implies that $T_{ij}$ remains positive definite for $t\in [0,T)$. Equivalently, there holds
\begin{equation}\label{s5:iii-1}
  \frac 1{\lambda_1}~\leq ~\frac 1{\varepsilon}f(\lambda)^{-1}=\frac 1{\varepsilon} f_*(\frac 1{\lambda_1},\cdots,\frac 1{\lambda_n})
\end{equation}
for all $t\in [0,T)$. Since $f_*$ approaches zero on the boundary of the positive cone $\Gamma_+$, the estimate \eqref{s5:iii-1} and Lemma 12 of \cite{Andrews-McCoy-Zheng} give the pinching estimate \eqref{s5:pinc-1}.

(iv). $n=2$. In this case, we don't need any second derivative condition on $F$. Define
\begin{equation*}
  G~=~\left(\frac{\lambda_2-\lambda_1}{\lambda_2+\lambda_1}\right)^2.
\end{equation*}
Then $G$ is homogeneous of degree zero of the shifted principal curvatures $\lambda_1,\lambda_2$. The evolution equation \eqref{s2:evl-S} implies that
 \begin{align}\label{s5:G3}
 \frac{\partial}{\partial t}G  =& \dot{F}^{kl}\nabla_k\nabla_lG+\left(\dot{G}^{ij}\ddot{F}^{kl,pq}-\dot{F}^{ij}\ddot{G}^{kl,pq}\right)\nabla_iS_{kl}\nabla^jS_{pq}\displaybreak[0]\nonumber\\ & -\left(\phi(t)+\sum_k\dot{f}^k\right)\dot{G}^{ij}S_{ik}S_{kj}+(\sum_k\dot{f}^k\lambda_k^2)\dot{G}^{ij}\delta_i^j.
 \end{align}
 The zero order terms of \eqref{s5:G3} are equal to
 \begin{equation*}
   Q_0=-4G\frac{\lambda_1\lambda_2}{\lambda_1+\lambda_2}\left(\phi(t)+\sum_k\dot{f}^k\right)-\frac{4G}{\lambda_1+\lambda_2}(\sum_k\dot{f}^k\lambda_k^2)~\leq~0.
 \end{equation*}
 The same argument as in \cite{And2010} gives that the gradient terms of \eqref{s5:G3} are non-positive at the critical point of $G$. Then the maximum principles implies that the supremum of $G$ over $M_t$ are non-increasing in time along the flow \eqref{flow-VMCF-2}. This gives the pinching estimate \eqref{s5:pinc-1} and the strict h-convexity of $M_t$ for all $t\in [0,T)$.
\endproof

\subsection{Shape estimate}
Denote by $\rho_-(t), \rho_+(t)$ the inner radius and outer radius of $\Omega_t$. Then there exists two points $p_1,p_2 \in \mathbb{H}^{n+1}$ such that $ B_{\rho_-(t)}(p_1)\subset \Omega_t\subset B_{\rho_+(t)}(p_2)$. By Corollary \ref{s5:cor}, the modified quermassintegral $\widetilde{W}_l$ is monotone under the inclusion of \emph{h-convex} domains in $\mathbb{H}^{n+1}$. This implies that
\begin{equation*}
\tilde{f}_l(\rho_-(t))=\widetilde{W}_l(B_{\rho_-(t)}(p_1))\leq \widetilde{W}_l(\Omega_t)\leq \widetilde{W}_l(B_{\rho_+(t)}(p_2))=\tilde{f}_l(\rho_+(t)).
\end{equation*}
Along the flow \eqref{flow-VMCF-2}, $\widetilde{W}_l(\Omega_t)=\widetilde{W}_l(\Omega_0)$ is a fixed constant. Therefore,  $$\rho_-(t)\leq C\leq \rho_+(t),$$where $C=\tilde{f}_l^{-1}(\widetilde{W}_l(\Omega_0))>0$ depends only on $l,n$ and $\Omega_0$.

On the other hand, since each $\Omega_t$ is \emph{h-convex}, the inner radius and outer radius of $\Omega_t$ satisfy $ \rho_+(t)~\leq ~c(\rho_-(t)+\rho_-(t)^{1/2})$ for some uniform positive constant $c$ (see  \cite{Cab-Miq2007,Mak2012}).  Thus there exist positive constants $c_1,c_2$ depending only on $n,l, M_0$ such that
\begin{equation}\label{s6:io-radius1}
  0<c_1\leq \rho_-(t)\leq \rho_+(t)\leq c_2
\end{equation}
for all time $t\in [0,T)$.

\subsection{$C^2$ estimate}
\begin{prop}\label{s6:F-ub}
Under the assumptions of Theorem \ref{thm1-5} with $\phi(t)$ given in \eqref{s1:phit-2},  we have $F\leq C$ for any $t\in [0,T)$, where $C$ depends on $n,l, M_0$ but not on $T$.
\end{prop}
\proof
For any given $t_0\in [0,T)$, let $B_{\rho_0}(p_0)$ be the inball of $\Omega_{t_0}$, where $\rho_0=\rho_-(t_0)$. Then a similar argument as in \cite[Lemma 4.2]{And-Wei2017-2} yields that
\begin{equation}\label{s6:inball-eqn1}
  B_{\rho_0/2}(p_0)\subset \Omega_t,\quad t\in [t_0, \min\{T,t_0+\tau\})
\end{equation}
for some positive $\tau$ depending only on $n,l,\Omega_0$. Consider the support function $u(x,t)=\sinh r_{p_0}(x)\langle \partial r_{p_0},\nu\rangle $ of $M_t$ with respect to the point $p_0$. Then the property \eqref{s6:inball-eqn1} implies that
\begin{equation}\label{s6:sup-1}
  u(x,t)~\geq ~\sinh(\frac{\rho_0}2)~=:~2c
\end{equation}
on $M_t$ for any $t\in[t_0,\min\{T,t_0+\tau\})$. On the other hand, the estimate \eqref{s6:io-radius1} implies that $u(x,t)\leq \sinh(2c_2)$ on $M_t$ for all $t\in[t_0,\min\{T,t_0+\tau\})$. Define the auxiliary function
\begin{equation*}
  W(x,t)=\frac {F(\mathcal{W}-\mathrm{I})}{u(x,t)-c}
\end{equation*}
on $M_t$ for $t\in [t_0,\min\{T,t_0+\tau\})$. Combining \eqref{s2:evl-F} and the evolution equation \eqref{s4:evl-u} for the support function, the function $W$ evolves by
\begin{align}\label{s6:evl-W-1}
   \frac{\partial}{\partial t}W= &\dot{F}^{ij}\left(\nabla_j\nabla_iW+\frac 2{u-c}\nabla_iu\nabla_jW\right)\nonumber \\
  &\quad-\frac{\phi(t)}{u-c}\left( \dot{F}^{ij}(h_i^kh_{k}^j-\delta_i^j)+W\cosh r_{p_0}(x)\right)\nonumber\\
  &\quad +\frac{F}{(u-c)^2}(F+\dot{F}^{kl}h_{kl})\cosh r_{p_0}(x)-\frac{cF}{(u-c)^2}\dot{F}^{ij}h_i^kh_{k}^j-W\dot{F}^{ij}\delta_i^j.
\end{align}
The second line of \eqref{s6:evl-W-1} involves the global term $\phi(t)$ and is clearly non-positive by the h-convexity of the evolving hypersurface. By the homogeneity of $f$ with respect to $\lambda_i=\kappa_i-1$, we have $F+\dot{F}^{kl}h_{kl}=2F+\sum_{k=1}^n\dot{f}^k$ and
\begin{equation*}
  \dot{F}^{ij}h_i^kh_{k}^j= \dot{f}^k(\lambda_k+1)^2=\dot{f}^k\lambda_k^2+2f+\sum_k\dot{f}^k~\geq Cf^2,
\end{equation*}
where the last inequality is due to the pinching estimate \eqref{s5:pinc-1}. The last term of \eqref{s6:evl-W-1} is non-positive and can be thrown away.  In summary, we arrive at
\begin{align*}
   \frac{\partial}{\partial t}W\leq & ~\dot{\Psi}^{ij}\left(\nabla_j\nabla_iW+\frac 2{u-c}\nabla_iu\nabla_jW\right)\nonumber \\
  &\quad +W^2(2+F^{-1}\sum_{k=1}^n\dot{f}^k)\cosh r_{p_0}(x)- c^2CW^3.
\end{align*}
Note that $\dot{f}^k$ is homogeneous of degree zero, the pinching estimate \eqref{s5:pinc-1} implies that each $\dot{f}^k$ is bounded from above and below by positive constants. Then without loss of generality we can assume that $F^{-1}\sum_{k=1}^n\dot{f}^k\leq 1$ since otherwise $F\leq \sum_{k=1}^n\dot{f}^k\leq C$ for some constant $C>0$. By the upper bound $r_{p_0}(x)\leq 2c_2$, we obtain the following estimate
\begin{align*}
   \frac{\partial}{\partial t}W\leq &~ \dot{\Psi}^{ij}\left(\nabla_j\nabla_iW+\frac 2{u-c}\nabla_iu\nabla_jW\right)+W^2\left(3\cosh (2c_2)- c^{2}CW\right)
\end{align*}
holds on $[t_0,\min\{T,t_0+\tau\})$. Then the maximum principle implies that $W$ is uniformly bounded from above and  the upper bound on $F$ follows by the upper bound on the outer radius in \eqref{s6:io-radius1}.
\endproof

\begin{prop}\label{s6:F-lb}
There exists a positive constant $C$, independent of time $T$, such that $F\geq C>0$.
\end{prop}
\proof
Since the evolving hypersurface $M_{t}$ is strictly h-convex, for each time $t_0\in [0,T)$ there exists a point $p\in \mathbb{H}^{n+1}$ and $x_0\in M_{t_0}$ such that $\Omega_{t_0}\subset B_{\rho_+(t_0)}(p)$ and $\Omega_{t_0}\cap B_{\rho_+(t_0)}(p)=x_0$.  By the estimate \eqref{s6:io-radius1} on the outer radius, the value of $F$ at the point $(x_0,t_0)$ satisfies
\begin{equation*}
  F(x_0,t_0)\geq \coth \rho_+(t_0)\geq \coth c_2.
\end{equation*}
Recall that the function $F$ satisfies the evolution equation \eqref{s2:evl-F} :
\begin{align}\label{s6:evl-F}
  \frac{\partial}{\partial t}F=&~g^{ik}\dot{F}^{ij}\nabla_k\nabla_jF+(F-\phi(t))(\dot{F}^{ij}h_i^kh_k^j-\dot{F}^{ij}\delta_i^j).
\end{align}
By the pinching estimate \eqref{s5:pinc-1} and the upper bound on the curvature proved in Proposition \ref{s6:F-ub}, the equation \eqref{s6:evl-F} is uniformly parabolic and the coefficient of the gradient terms and the lower order terms in \eqref{s6:evl-F} have bounded $C^0$ norm. Then there exists $r>0$ depending only on the bounds on the coefficients of \eqref{s6:evl-F} such that we can apply the Harnack inequality of Krylov and Safonov \cite{KS81} to \eqref{s6:evl-F} in a  space-time neighbourhood $B_r(x_0)\times (t_0-r^2,t_0]$ of $x_0$ and deduce the lower bound $F\geq C F(x_0,t_0)\geq C>0$ in a smaller neighbourhood $B_{r/2}(x_0)\times (t_0-\frac 14r^2,t_0]$. Note that the diameter $r$ of the space-time neighbourhood is not dependent on the point $(x_0,t_0)$. Consider the boundary point $x_1\in \partial B_{r/2}(x_0)$. We can look at the equation \eqref{s6:evl-F} in a neighborhood $B_r(x_1)\times (t_0-r^2,t_0]$ of the point $(x_1,t_0)$. The Harnack inequality implies that $F\geq C F(x_1,t_0)\geq C>0$ in $B_{r/2}(x_1)\times (t_0-\frac 14r^2,t_0]$.  Since the diameter of each $M_{t_0}$ is uniformly bounded from above, after a finite number of iterations, we conclude that $F\geq C>0$ on $M_{t_0}$ for a uniform constant $C$ independent of $t_0$.
\endproof

The pinching estimate \eqref{s5:pinc-1} together with the bounds on $F$ proven in Proposition \ref{s6:F-ub} and Proposition \ref{s6:F-lb} implies that the shifted principal curvatures $\lambda=(\lambda_1,\cdots,\lambda_n)$ satisfy
\begin{equation*}
  0<C^{-1}~\leq ~\lambda_i~\leq ~C
\end{equation*}
for some constant $C>0$ and $t\in [0,T)$. This gives the uniform $C^2$ estimate of the evolving hypersurfaces $M_t$. Moreover, the global term $\phi(t)$ given in \eqref{s1:phit-2} satisfies $0<C^{-1}\leq \phi(t)\leq C$ for some constant $C>0$.

\subsection{Long time existence and convergence}

If $F$ is inverse-concave, by applying the similar argument as \cite{And-Wei2017-2,Mcc2017} (see also \cite{TW2013}) to the equation \eqref{s5:flow-u},  we can first derive the $C^{2,\alpha}$ estimate and then the $C^{k,\alpha}$ estimate for all $k\geq 2$. If $F$ is concave or $n=2$, we write the flow \eqref{flow-VMCF-2} as a scalar parabolic PDE for the radial function as follows: Since each $M_t$ is strictly h-convex, we write $M_{t}$ as a radial graph over a geodesic sphere for a smooth function $\rho$ on $S^n$. Let $\{\theta^i\}, i=1,\cdots,n$ be a local coordinate system on $S^n$. The induced metric on $M_{t_0}$ from $\mathbb{H}^{n+1}$ takes the form
$$  g_{ij}=\bar{\nabla}_i\rho \bar{\nabla}_j\rho+\sinh^2\rho\bar{g}_{ij},$$
where $\bar{g}_{ij}$ denotes the round metric on $S^n$. Up to a tangential diffeomorphism, the flow equation \eqref{flow-VMCF-2} is equivalent to the following scalar parabolic equation
\begin{equation}\label{s6:graph-flow}
  \frac{\partial }{\partial t}\rho=~(\phi(t)-F(\mathcal{W}-\mathrm{I}))\sqrt{1+{|\bar{\nabla}\rho|^2}/{\sinh^2\rho}}.
\end{equation}
for the smooth function $\rho(\cdot,t)$ on $S^n$. The Weingarten matrix $\mathcal{W}=(h_i^j)$ can be expressed as
\begin{equation*}
  h_i^j=~\frac{\coth \rho}{v}\delta_i^j+\frac{\coth \rho}{v^3\sinh^2 \rho}\bar{\nabla}^j\rho\bar{\nabla}_i\rho-\frac {\tilde{\sigma}^{jk}}{v\sinh^2\rho}\bar{\nabla}_{k}\bar{\nabla}_i\rho,
\end{equation*}
where
 \begin{equation*}
  v=\sqrt{1+{|\bar{\nabla}\rho|^2}/{\sinh^2\rho}},\quad \mathrm{and }\quad  \tilde{\sigma}^{jk}~=~\sigma^{jk}-\frac{\bar{\nabla}^j\rho\bar{\nabla}^k\rho}{v^2\sinh^2\rho}.
\end{equation*}
Thus we can apply the argument as in \cite{And2004,McC2005} to derive the higher regularity estimate. Therefore, for any $F$ satisfying the assumption of Theorem \ref{thm1-5}, the solution of the flow \eqref{flow-VMCF-2} exists for all time $t\in[0,\infty)$ and remains smooth and strictly h-convex. Moreover, the Alexandrov reflection argument as in \cite[\S 6]{And-Wei2017-2} implies that the flow converges smoothly as time $t\to\infty$ to a geodesic sphere $\partial B_{r_{\infty}}$ which satisfies $\widetilde{W}_l(B_{r_{\infty}})=\widetilde{W}_l(\Omega_{0})$.  This finishes the proof of Theorem \ref{thm1-5}.

\section{Conformal deformation in the conformal class of $\bar g$}

In this section we mention an interesting connection (closely related to the results of \cite{EGM}) between flows of h-convex hypersurfaces in hyperbolic space by functions of principal curvatures, and conformal flows of conformally flat metrics on $S^n$.  This allows us to translate some of our results to convergence theorems for metric flows, and our isoperimetric inequalities to corresponding results for conformally flat metrics.

The crucial observation is that there is a correspondence between conformally flat metrics on $S^n$ satisfying a certain curvature inequality, and horospherically convex hypersurfaces.  To describe this, we recall that the isometry group of $\mathbb H^{n+1}$ coincides with $O_+(n+1,1)$, the group of future-preserving linear isometries of Minkowski space.  This also coincides with the M\"obius group of conformal diffeomorphisms of $S^n$, by the following correspodence:  If $L\in O_+(n+1,1)$, we define a map $\rho_L$ from $S^n$ to $S^n$ by
$$
\rho_L(\e) = \pi(L(\e,1)),
$$
where $\pi(x,y)=\frac{x}{y}$ is the radial projection from the future null cone to the sphere at height $1$.
This defines a group homomorphism from $O_+(n+1,1)$ to the group of M\"obius transformations.  We have the following result:

\begin{prop}
If $L\in O_+(n+1,1)$ and $M\subset{\mathbb H}^{n+1}$ is a horospherically convex hypersurface with horospherical support function $u:\ S^n\to\RR$, denote by $u_L$ the horospherical support function of $L(M)$.  Then $\rho_L$ is an isometry from $\E^{-2u}\bar g$ to $\E^{-2u_L}\bar g$.  That is,
$$
\E^{-2u(\e)}\bar g_{\e}(v_1,v_2) = \E^{-2u_L(\rho_L(\e))}\bar g_{\rho_L(\e)}(D\rho_L(v_1),D\rho_L(v_2))
$$
for all $\e\in S^n$ and $v_1,v_2\in T_\e S^n$.
\end{prop}

\begin{proof}
We compute:
\begin{align*}
\E^{-u(\e)}&= -X\cdot (\e,1)\\
&= -L(X)\cdot L(\e,1)\\
&= - L(X)\cdot \mu (\e_L,1)\qquad\text{where\ }\mu = |L(\e,1)\cdot(0,1)|\\
&=\mu \E^{-u_L(\e_L)}.
\end{align*}
On the other hand the M\"obius transformation $\rho_L$ is a conformal transformation with conformal factor
$\mu=|L(\e,1)\cdot(0,1)|$.  The result follows directly.
\end{proof}

\begin{cor}
Isometry invariants of a horospherically convex hypersurface $M$ are M\"obius invariants of the conformally flat metric
$\tilde g=\E^{-2u}\bar g$, and vice versa.  In particular, Riemannian invariants of $g$ are isometry invariants of $M$.
\end{cor}

Computing explicitly, we find that for $n>2$ the Schouten tensor
$$\tilde S_{ij} = \frac{1}{n-2}\left(\tilde R_{ij}-\frac{\tilde R}{2(n-1)}\tilde g_{ij}\right)$$ of $\tilde g$ (which completely determines the curvature tensor for a conformally flat metric) is given by
\begin{align*}
\tilde S_{ij} &= \frac12\bar g_{ij}+\bar\nabla_i\bar\nabla_ju+u_iu_j-\frac12|\bar\nabla u|^2\bar g_{ij}\\
&= \E^{-u}A_{ij}+\frac12 \tilde{g}_{ij}\\
&= \left[(\mathcal{W}-\mathrm I)^{-1}\right]_i^p \tilde g_{pk}+\frac12 \tilde g_{ij}.
\end{align*}
It follows that the eigenvalues of $\tilde S_{ij}$ (with respect to $\tilde g_{ij}$) are $\frac12 + \frac{1}{\lambda_i}$, where $\lambda_i = \kappa_i-1$.    When $n=2$ the tensor $\tilde{S}_{ij}$ defined by the right-hand side of the above equation is by construction M\"obius-invariant, and so gives a M\"obius invariant of $\tilde g$ which is not a Riemannian invariant.  This tensor still has the same relation to the principal curvatures of the corresponding h-convex hypersurface.

We observe that this connection between the Schouten tensor of $\tilde g$ and the Weingarten map of the hypersurface leads to a conincidence between the corresponding evolution equations:  If a family of h-convex hypersurfaces $M_t=X(M,t)$ evolves according to a curvature-driven evolution equation of the form
$$
\frac{\partial X}{\partial t} = -F({\mathcal W}-{\mathrm I},t)\nu
$$
then the metric $\tilde g$ satisfies $\tilde S>\frac12\tilde g$, and evolves according to the parabolic conformal flow
$$
\frac{\partial\tilde g}{\partial t} = 2F((\tilde S-\frac12\tilde g)^{-1},t)\tilde g.
$$
In particular the convergence theorems for hypersurface flows correspond to convergence theorems for the corresponding conformal flows, and the resulting geometric inequalities for hypersurfaces imply corresponding geometric inequalities for the metric $\tilde g$.

\end{document}